\documentclass[journal]{IEEEtran}

\usepackage{cite}
\ifCLASSINFOpdf
\usepackage[pdftex]{graphicx}
\graphicspath{{../pdf/}{../jpeg/}}
\DeclareGraphicsExtensions{.pdf,.jpeg,.png}
\else
\fi

\usepackage[cmex10]{amsmath}

\usepackage{amssymb, bm}

\DeclareMathOperator*{\argmin}{arg\,min}
\usepackage{comment}
\usepackage{array,multirow}
\usepackage{tikz}
\usepackage[noabbrev,capitalise]{cleveref}

\usepackage{accents}

\usepackage{bbm}
\usepackage{dsfont}
\usepackage{bbold}
\usepackage{dblfloatfix}

\usepackage{algorithm}
\usepackage{algorithmic}
\usepackage{enumerate}
\usepackage{amsthm}
\theoremstyle{definition}

\newtheorem{theorem}{Theorem}

\newtheorem{corollary}{Corollary}
\newtheorem{definition}{Definition}

\newtheorem{remark}{Remark}
\newtheorem*{problem formulation}{Problem formulation}

\usepackage{enumitem}

\usepackage{dsfont}

\usepackage{amsfonts}

\usepackage{mathtools}

\usepackage{color,soul}

\begin{document}
\bstctlcite{IEEEexample:BSTcontrol}

\title{Feasible Path Identification in Optimal Power Flow with Sequential Convex Restriction
}
\author{Dongchan Lee, Konstantin Turitsyn, Daniel K. Molzahn, and Line A. Roald
\thanks{
This work was supported in part by the U.S. Department of Energy Office of Electricity as part of the DOE Grid Modernization Initiative and in part by the National Science Foundation Energy, Power, Control and Networks Award 1809314.

D. Lee is with the Department of Mechanical Engineering, Massachusetts Institute of Technology, Cambridge, MA 02139, USA (email: dclee@mit.edu).

K. Turitsyn is with the D. E. Shaw Group, New York, NY 10036 USA (e-mail: turitsyn@mit.edu).

D. K. Molzahn is with the School of Electrical and Computer Engineering, Georgia Institute of Technology, Atlanta, GA 30313 USA (e-mail: molzahn@gatech.edu).

L. A. Roald is with the Department of Electrical and Computer Engineering, University of Wisconsin, Madison, WI 53706 USA (e-mail: roald@wisc.edu).
}
}
\maketitle

\begin{abstract}
Nonconvexity induced by the nonlinear AC power flow equations challenges solution algorithms for AC optimal power flow (OPF) problems. 
While significant research efforts have focused on reliably computing high-quality OPF solutions, it is not always clear that there exists a feasible path to reach the desired operating point.
Transitioning between operating points while avoiding constraint violations can be challenging since the feasible space of the OPF problem is nonconvex and potentially disconnected.
To address this problem, we propose an algorithm that computes a provably feasible path from an initial operating point to a desired operating point.
Given an initial feasible point, the algorithm solves a sequence of convex quadratically constrained optimization problems over conservative convex inner approximations of the OPF feasible space. In each iteration, we obtain a new, improved operating point and a feasible transition from the operating point in the previous iteration. 
In addition to computing a feasible path to a known desired operating point, this algorithm can also be used to improve the operating point locally. 
Extensive numerical studies on a variety of test cases demonstrate the algorithm and the ability to arrive at a high-quality solution in few iterations.
\end{abstract}

\begin{IEEEkeywords}
Feasible Path Identification, Convex Restriction, Optimal Power Flow
\end{IEEEkeywords}

\IEEEpeerreviewmaketitle

\section{Introduction}
AC optimal power flow (OPF) is a fundamental optimization problem in power system analysis \cite{carpentier1962,stott2012, capitanescu2011stateoftheart}.
The classical form of an OPF problem seeks an operating point that is \emph{feasible} (i.e., satisfies both the AC power flow equations that model the network physics and the inequality constraints associated with operational limits on voltage magnitudes, line flows, generator outputs, etc.) and economically efficient (i.e., minimizes generation cost). 
While previous research has improved the tractability of OPF algorithms and the quality of the resulting solutions~\cite{opf_litreview1993IandII,ferc4,capitanescu2011stateoftheart,Frank2012,stott2012,pscc2014survey,abdi2017,molzahn2018fnt}, a number of challenging issues remain.
One such issue is to determine a \emph{sequence} of control actions that facilitate a safe transition from the current operating point to the desired operating point \cite{stott2012, capitanescu2011stateoftheart}.

With increasing variability due to the growth of renewable generation and the expansion of electricity markets, power systems experience more frequent and larger magnitude transitions from one operational state to another.
When there is a significant change in the operating point, the transition needs to be carried out gradually in a sequential order rather than instantaneously.
Previous literature has considered sequencing control actions to mitigate disturbances \cite{hong1993, Capitanescu2011, phan2014minimal}. References such as \cite{Capitanescu2011, phan2014minimal} formulate a mixed-integer linear programming problem to minimize the \emph{number} of control actions based on linear approximations of the AC power flow equations.
While these linear approximations avoid complications from nonlinear equality constraints, they do not guarantee feasibility of the full, non-linear AC OPF problem.
Approximations of the power flow equations may result in an infeasible state and can lead to incorrect security assessments.
This can be understood by analyzing the set of feasible dispatch points (also referred to as the feasible space) for the AC OPF problem, which is nonconvex and sometimes disconnected \cite{bukhsh_tps,Molzahn2017,Lee2018}.
Due to the complicated feasible space, a \emph{feasible path} connecting the two steady-state operating points (where each intermediate state is AC power flow feasible) can be difficult to compute or may not exist. The example shown in Figure~\ref{fig:intro} illustrates a situation where a linear transition from an initial to a desired operating point leads to constraint violations, but the piece-wise linear transition shown by the red line allows us to reach the desired operating point.

\begin{figure}[!b]
	\centering
	\includegraphics[width=2.8in]{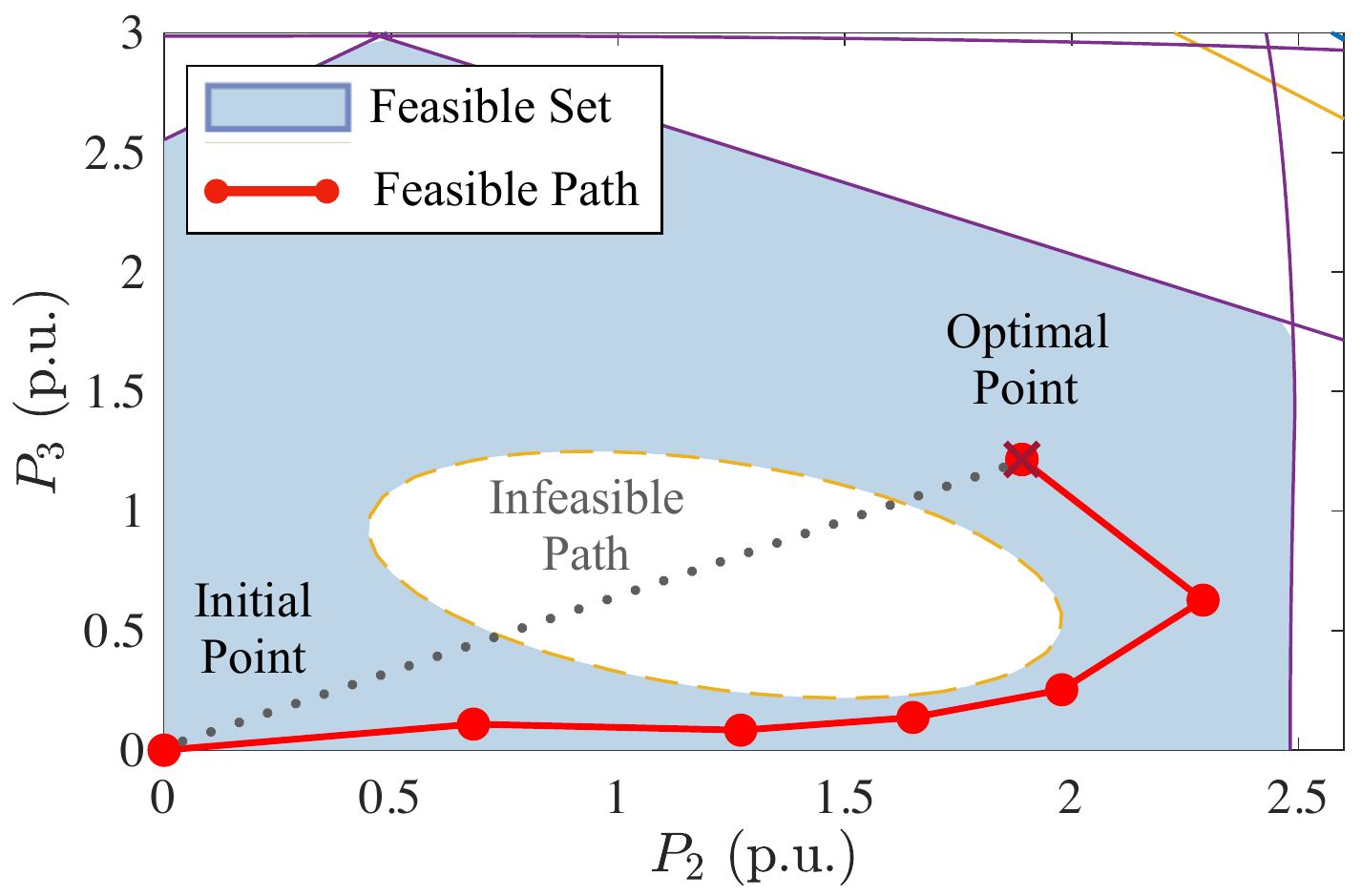}
	\caption{Illustration of a feasible path identification for a 9-bus system. The axes show the active power generation at buses 2 and 3. The blue region is the nonconvex feasible space, which is defined by the line flow limits (in purple) and the reactive power generation limits (in yellow). The black dotted line illustrates how a linear transition from the current to the desired operating point crosses through an infeasible region. The red line provides a feasible, piece-wise linear path from the origin to the desired operating point that avoids the infeasible region.}
	\label{fig:intro}
\end{figure}

The main contribution of this paper is an algorithm that allows us to compute a feasible path between two operating points. For the purposes of the paper, we are interested in \emph{steady-state feasibility}, i.e., feasibility of the nonlinear AC power flow equations and operational limits without consideration of the system dynamics. 

The idea of a feasible path may bring to mind the concept of \emph{numerical continuation}, which produces a sequence of solutions to a system of nonlinear equations when the parameters are varied along one degree of freedom~\cite{garcia81a}. 
Common applications of continuation methods include computing voltage stability margins~\cite{continuation,ieee02a,milano2010}, analyzing the power flow solvability boundary~\cite{hiskens_boundary}, and calculating multiple power flow solutions~\cite{thorp1993,molzahn_lesieutre_chen-counterexample,lesieutre2015}. Continuation methods have also been used to improve the convergence characteristics of solution algorithms for OPF problems, particularly near the voltage stability boundary~\cite{ponrajah1989,huneault1990investigation,cvijic2012,Bie18Toward}, and have additionally been applied to compute multiple local optima for OPF problems~\cite{wu_molzahn_lesieutre_dvijotham-acopf_tracing}. 
Another application of continuation methods in the context of OPF problems is the tracking of optimal dispatch points with respect to varying load parameters. This is achieved by applying continuation methods to trace the solutions to the KKT optimality conditions for the OPF problem\mbox{\cite{huneault1985continuation,almeida1994,ajjarapu1995optimal,almeida2000,huneault1990investigation}.}

While the idea of tracking the optimal solutions might seem similar to the idea of a feasible path proposed in this paper, the goal of this problem formulation is fundamentally different from our problem. 
The continuation OPF identifies a sequence of optimal solutions, frequently under limiting assumptions on nonlinear operational constraints such as line flow limits and reactive power generation limits. 
There is no guarantee that the transition between subsequent points obtained with the continuation method are feasible, and therefore the optimal ``path'' of the continuation method consists of a set of discrete, optimal points with respect to a varying load profile. 
In contrast, our method identifies a continuous path where every transition is guaranteed to be feasible with respect to both the full, non-linear AC power flow equations and other operational constraints.
%
%The optimality and feasibility of the dispatch points are granted for the discrete sequence of points, but not during the \emph{transition} between the consecutive dispatch points produced by the continuation method. In this paper, the feasible path ensures feasibility for not only the discrete points in the sequence but also the dispatch points during the transition.
%
Another distinction is that continuation algorithms which trace an OPF solution require that the initial operating point is the solution to an OPF problem for some load profile. Our approach specifically does not require an optimal starting point, but rather permits the calculation of a feasible transition path between any feasible operating points.

Another related research area is model predictive control, 
which has been applied to a variety of power system applications such as frequency control~\cite{mpc_nerc,mpc_agc}, voltage regulation~\cite{moradzadeh2012}, HVDC control~\cite{fuchs2013}, contingency mitigation~\cite{martin2016}, etc.
MPC approaches generally rely on linearized dynamics without considering the nonlinear AC power flow equations, and hence can not
guarantee feasibility.
Additional related work in~\cite{Bernstein2015} and~\cite{Wang2017} controls the power injections in order to enforce operational constraints in distribution systems. However, similar to MPC, these methods are focused more on real-time control as opposed to optimization. Further, these approaches are limited to radial networks with only PQ buses, making them unsuitable for transmission system applications.

We finally mention another class of related algorithms called sequential convex programming, which solve optimization problems via iteratively updated convex approximations of the equations over a trust region~\cite{conn2000trust,nocedal2006numerical}. Sequential convex programming approaches such as~\cite{zamzam2017} and~\cite{wei2017} seek OPF solutions by solving a sequence of convex power flow approximations. However, these convex approximations are not inner approximations of the feasible space and therefore do not guarantee feasibility of the intermediate iterates or the transition between them.

Thus, the main novelty of our method is that we provide a dispatch solution and \emph{an associated transition path} for which the nonlinear AC power flow equations are guaranteed to be feasible and the operational constraints are certifiably satisfied.
To the best of our knowledge, this paper proposes the first algorithm that provides guaranteed feasible paths for general OPF problems.

Specifically, we propose an algorithm for computing a sequence of control actions that ensures feasibility with respect to both the nonlinear AC power flow equations and operational limits (in the form of inequality constraints) as the system transitions from one operating point to another.
In contrast to previous work, our proposed feasible path algorithm is not limited to specific classes of systems, considers the nonlinear AC power flow model, and is tractable for large problems. Based on a quadratic convex restriction of the AC power flow feasible space~\cite{Lee2018}, we compute a piece-wise linear path connecting an initial point to a desired operating point such that all points along the path are feasible. 

We highlight two characteristics of the convex restriction which are crucial for the success of our algorithm:\\
\hspace*{5pt} (i) The convex restriction provides a \emph{conservative inner approximation} of the feasible space of the AC power flow equations, which implies that all points within the restriction are AC power flow feasible (and, by proper extensions, feasible for additional inequality constraints). This is in contrast to convex relaxations, which extend the originally nonconvex feasible space to become convex by adding infeasible points.\\
\hspace*{5pt} (ii) The convex restriction is, as the name implies, a convex set. This means that the transition between any two points within the convex restriction will also lie inside the convex restriction, hence guaranteeing that there exists a feasible AC power flow solution at any intermediate point.

In summary, the main contributions of this paper are:
\begin{enumerate}
\item We formulate the AC OPF problem based on convex restrictions from \cite{Lee2018}. This is the first formulation of an OPF with convex restriction, which requires extending the convex restriction to include line flow constraints.
The formulation guarantees that the linear transition between any two points within the restriction is feasible, is applicable to general system models, considers the nonlinear AC power flow equations, and is tractable for large problems.
\item Using the OPF with convex restriction, we propose a sequential algorithm which in each iteration (i) constructs a convex restriction around a feasible point and (ii) solves the OPF problem to obtain an improved feasible point. The algorithm's outcome is a piece-wise linear feasible path. We provide two objective functions which either achieve local improvements to the current operating point or identify a feasible path to a desired operating point. 
\item We demonstrate the capabilities of the algorithm using numerical experiments on a variety of test cases.
\end{enumerate}

The rest of the paper is organized as follows. Section~\ref{sec:model} presents the system model and preliminaries. Section~\ref{sec:restriction} reviews and extends convex restriction techniques to formulate the OPF problem, including line flow limits and other features. Section~\ref{sec:path} presents our algorithm for computing OPF solutions with corresponding feasible paths. Section~\ref{sec:experiments} demonstrates the proposed algorithm with numerical experiments and illustrative figures. Section~\ref{sec:conclusion} concludes the paper.

\section{System Model and Preliminaries}
\label{sec:model}

Consider a power network with sets of buses $\mathcal{N}$ and lines $\mathcal{E}\subseteq\mathcal{N}\times \mathcal{N}$. The scalars $n_b$, $n_g$, $n_{pq}$, and $n_l$ denote the number of buses, generators, PQ buses, and lines.
%The admittance matrix is $Y\in\mathbb{C}^{n_b\times n_b}$, with associated conductance and susceptance matrices $G$ and $B$ such that $Y=G+jB$.
The network's incidence matrix is $E\in\mathbb{R}^{n_b\times n_l}$.
The connection matrix between generators and buses is $C\in\mathbb{R}^{n_b\times n_g}$, where the $(i,k)$ element of $C$ is equal to 1 for each bus $i$ and generator $k$ and zero otherwise. The active and reactive power generations are $p_\textrm{g}\in\mathbb{R}^{n_g}$ and $q_\textrm{g}\in\mathbb{R}^{n_g}$. Specified values of active and reactive load demands are denoted $p_\textrm{d}\in\mathbb{R}^{n_b}$ and $q_\textrm{d}\in\mathbb{R}^{n_b}$. The buses' voltage magnitudes and phase angles are $v\in\mathbb{R}^{n_b}$ and $\theta\in\mathbb{R}^{n_b}$.
The superscripts ``$\textrm{f}$'' and ``$\textrm{t}$'' denote \textit{from} and \textit{to} buses for the lines. The subscripts ``$\textrm{v}\theta$'', ``$\textrm{ns}$'', ``$\textrm{pv}$'', and ``$\textrm{pq}$'' denote the slack (V$\theta$), non-slack (non-V$\theta$), PV, and PQ elements of the corresponding vector. Superscript ``$T$'' denotes the transpose. $I$ and $\mathbf{0}$ denote identity and zero matrix of appropriate size.

\subsection{Phase-Adjusted AC Power Flow Formulation}
To set the stage for our further discussion, we describe a slightly modified representation of the standard AC power flow equations, the so-called \emph{phase-adjusted AC power flow formulation}.
The formulation is defined relative to a known, feasible \emph{base point} denoted by the subscript $0$, i.e., $v_0$ and $\theta_0$ denote the base point's voltage magnitude and phase angle. 

The angle differences across each line are
\begin{equation}
	\varphi_l=\theta_l^\textrm{f}-\theta_l^\textrm{t}, \hskip2em l=1,\ldots,n_l,
\end{equation}
where $\theta_l^\textrm{f}$ and $\theta_l^\textrm{t}$ denote the phase angle of the \textit{from} bus and \textit{to} bus of line $l$. This can be equivalently expressed as $\varphi=E^T\theta$. The \emph{phase-adjusted} angle differences are then defined as
\begin{equation}
    \tilde{\varphi}=\varphi-\varphi_0=E^T\theta-E^T\theta_0= E^T(\theta-\theta_0), \nonumber
\end{equation}
where $\varphi_0$ is the base phase angle differences. 
With this, the phase-adjusted AC Power Flow equations can be written for each bus $k=1,\ldots,n_b$,
\begin{subequations} \begin{align}
p^\textrm{inj}_k=\sum_{l=1}^{n_l} v^\textrm{f}_lv^\textrm{t}_l\left(\widehat{G}^c_{kl}\cos{\tilde{\varphi}_l}+\widehat{B}^s_{kl}\sin{\tilde{\varphi}_l} \right)+G_\mathit{kk}^dv_k^2,\\
q^\textrm{inj}_k=\sum_{l=1}^{n_l} v^\textrm{f}_lv^\textrm{t}_l\left(\widehat{G}^s_{kl}\sin{\tilde{\varphi}_l}-\widehat{B}^c_{kl}\cos{\tilde{\varphi}_l} \right)-B_\mathit{kk}^dv_k^2,
\end{align} \label{eqn:DACPA} \end{subequations}

% The equation is expressed by the power flow across individual transmission lines in terms of a set of non-linear \emph{basis functions} which involve only the voltage magnitudes and angles at each end of the line.

\noindent where $v^\textrm{f}_l$ and $v^\textrm{t}_l$ denote the voltage magnitude at the \textit{from} and \textit{to} buses of line $l$. The active and reactive power injections are
$p^\textrm{inj}=Cp_\textrm{g}-p_{\textrm{d}}$ and $q^\textrm{inj}=Cq_\textrm{g}-q_{\textrm{d}}$. The matrices $\widehat{G}^\textrm{c},\ \widehat{G}^\textrm{s},\ \widehat{B}^\textrm{c},\ \widehat{B}^\textrm{s}\in\mathbb{R}^{n_b\times n_l}$ and $G^d,\ B^d\in\mathbb{R}^{n_b\times n_b}$ are phase-adjusted admittance matrices defined relative to the base point, and their derivations are shown in the Appendix. In addition to the power flow equations in \eqref{eqn:DACPA}, the OPF problem enforces the following operational constraints:
\begin{subequations} \begin{align}
&p_{\textrm{g},i}^\textrm{ min}\!\leq\! p_{\textrm{g},i}\!\leq\! p_{\textrm{g},i}^\textrm{max}, \ q_{\textrm{g},i}^\textrm{min}\!\leq\! q_{\textrm{g},i}\!\leq\! q_{\textrm{g},i}^\textrm{max}, \ & i&\!=\!1,\ldots,n_g, \label{eqn:pgenlim} \\
&v_j^{\textrm{min}}\!\leq\! v_j\!\leq\! v_j^{\textrm{max}}, \ &j&\!=\!1,\ldots,n_b \label{eqn:vlim}\\ 
& \varphi^{\textrm{min}}_l\!\leq\! \varphi_l\!\leq\! \varphi^{\textrm{max}}_l, \ & l&\!=\!1,\ldots,n_l, \label{eqn:anglelim} \\
&(s_{p,l}^\textrm{f})^2+(s_{q,l}^\textrm{f})^2\leq (s^{\textrm{max}}_l)^2, \ & l&\!=\!1,\ldots,n_l, \label{eqn:linelim_f} \\ &(s_{p,l}^\textrm{t})^2+(s_{q,l}^\textrm{t})^2\leq (s^{\textrm{max}}_l)^2, \ & l&\!=\!1,\ldots,n_l. \label{eqn:linelim_t}
\end{align}%
\label{eqn:OPconstr}%
\end{subequations}%
Equation \eqref{eqn:pgenlim} represents the generators' active and reactive power capacity limits,  $p_g^{\text{max}}, p_g^{\text{min}}$ and $q_g^{\text{max}}, q_g^{\text{min}}$, respectively. Equations~\eqref{eqn:vlim} and \eqref{eqn:anglelim} limit the voltage magnitudes to the range $v^{\text{min}}, v^{\text{max}}$ and enforce stability limits on the angle differences $\varphi^{\text{min}}, \varphi^{\text{max}}$. Equations \eqref{eqn:linelim_f} and \eqref{eqn:linelim_t} impose the line capacity limit $s^\textrm{max}$ where
$s_{p,l}^\textrm{f},\ s_{q,l}^\textrm{f}$ represent the active and reactive power flowing into the line~$l$ at the \textit{from} buses, and $s_{p,l}^\textrm{t}$, $s_{q,l}^\textrm{t}$ represent the active and reactive power flowing into the line $l$ at the \textit{to} buses. 

The phase-adjusted AC power flow equations can be expressed in terms of basis functions, which are defined as
\begin{equation} \begin{aligned}
\psi_l^\mathrm{C}(v,\varphi)&=v^\mathrm{f}_lv^\mathrm{t}_l\cos{(\varphi_l-\varphi_{0,l})}, \ & l&=1,\ldots,n_l, \\
\psi_l^\mathrm{S}(v,\varphi)&=v^\mathrm{f}_lv^\mathrm{t}_l\sin{(\varphi_l-\varphi_{0,l})}, \ & l&=1,\ldots,n_l, \\
\psi_k^\mathrm{Q}(v,\varphi)&=v_k^2, \ & k&=1,\ldots,n_b.
\label{eqn:basisfunc}
\end{aligned} \end{equation}
%{\color{red}I know, ugly notation... Maybe you aeva  better suggestion here...}
The power flow equations~\eqref{eqn:DACPA} can then be rewritten as
\begin{equation}
\begin{bmatrix} Cp_\textrm{g}-p_\textrm{d} \\[+2pt] Cq_\textrm{g}-q_\textrm{d} \end{bmatrix}
+
\begin{bmatrix}
-\widehat{G}^\textrm{c} & -\widehat{B}^\textrm{s} & -G^\mathrm{d} \\
\widehat{B}^\textrm{c} & -\widehat{G}^\textrm{s} & B^\mathrm{d}
\end{bmatrix} \psi(v,\varphi)=0,
\label{eqn:pf_vec}\end{equation}
where $\psi(v,\varphi)=\begin{bmatrix}\psi^\mathrm{C}(v,\varphi)^T & \psi^\mathrm{S}(v,\varphi)^T & \psi^\mathrm{Q}(v,\varphi)^T\end{bmatrix}^T$.

\subsection{Control and State Variables}
Standard power system definitions divide the system into three sets of buses:
\begin{itemize}
    \item PV buses: $p_\textrm{pv}^\textrm{inj},\,v_\textrm{pv}$ specified; $q_\textrm{pv}^\textrm{inj},\,\theta_\textrm{pv}$ implicitly defined.
    \item PQ buses: $p_\textrm{pq}^\textrm{inj},\,q_\textrm{pq}^\textrm{inj}$ specified; $v_\textrm{pq},\,\theta_\textrm{pq}$ implicitly defined.
    \item V$\theta$ (slack) bus: $v_{\textrm{v}\theta},\,\theta_{\textrm{v}\theta}$ specified; $p_{\textrm{v}\theta}^\textrm{inj},\,q_{\textrm{v}\theta}^\textrm{inj}$ implicitly defined.
\end{itemize}
For the analysis, variables that are explicitly set by the system operator are control variables, and variables that are implicitly determined through the AC power flow equations are state variables. 
Constants such as the active and reactive power load on PQ buses and the reference angle $\theta_{\textrm{v}\theta}\!=\!0$ are not considered as variables.

The \emph{control} variables are the active power outputs of generators at PV buses, $p_\textrm{pv}$, and the voltage magnitudes at the V$\theta$ and PV buses, $v_\textrm{g}$. The \emph{state} variables are 
the phase angles at non-slack buses, $\theta_\textrm{ns}$, and voltage magnitude at PQ buses, $v_\textrm{pq}$. The state variables are implicitly defined through the power flow equations given a set of control variables.
The set of control variables are denoted by $u\in\mathbb{R}^{2n_g-1}$, and state variables are denoted by $x\in\mathbb{R}^{n_b-1+n_{pq}}$ where
\begin{equation}
u=\begin{bmatrix} p_{\textrm{pv}} \\ v_\textrm{g}\end{bmatrix}, \hskip2em x=\begin{bmatrix} \theta_\textrm{ns} \\ v_\textrm{pq}\end{bmatrix}.
\end{equation}

In addition, the \emph{intermediate} variables $[p_{\textrm{v}\theta}, q_{\textrm{v}\theta}, q_\textrm{pv}]$ are explicitly defined by the power flow equations and a given set of state and control variables $(x,u)$.

 For a given set of control variables $u$, the state variables $x$ are obtained from a subset of the phase-adjusted power flow equations \eqref{eqn:pf_vec},
\begin{equation}
\underbrace{\begin{bmatrix}
C_\textrm{ns}p_\textrm{g}-p_\textrm{d,ns} \\ C_\textrm{pq}q_\textrm{g}-q_\textrm{d,pq}\end{bmatrix}}_{\tau(u)}
+\underbrace{\begin{bmatrix}
-\widehat{G}^\textrm{c}_\textrm{ns} & -\widehat{B}^\textrm{s}_\textrm{ns} & -G^\mathrm{d}_\textrm{ns} \\
\widehat{B}^\textrm{c}_\textrm{pq} & -\widehat{G}^\textrm{s}_\textrm{pq} & B^\mathrm{d}_\textrm{pq}
\end{bmatrix}}_{M_\textrm{eq}}\psi(v,\varphi)=0,
\label{eqn:pf_eqn}
\end{equation}
where $\tau(u)$ denotes the active and reactive power injections at certain buses. The matrix $\widehat{G}^\textrm{c}_\textrm{ns}\in\mathbb{R}^{(n_b-1)\times n_l}$ contains the rows corresponding to the non-slack buses from $\widehat{G}^\textrm{c}\in\mathbb{R}^{n_b\times n_l}$, and $\widehat{B}^\textrm{c}_\textrm{pq}\in\mathbb{R}^{n_{pq}\times n_l}$ contains the rows corresponding to PQ buses from $\widehat{B}^\textrm{c}\in\mathbb{R}^{n_b\times n_l}$. The other submatrices are defined similarly. Note that~\eqref{eqn:pf_eqn} is a square system of equations where the number of state variables and the number of equations are the same. 
This is a minimal subset of the AC power flow equations that completely describes the relationships among the control and state variables.

The intermediate variables (i.e., the active power at the V$\theta$ bus $p_{\textrm{v}\theta}$ and the reactive power at the V$\theta$ and PV buses $q_{\textrm{v}\theta}, q_{\textrm{pv}}$) are functions of state and control variables $(x,u)$:
\begin{equation}
\underbrace{\begin{bmatrix} 
C_{\textrm{v}\theta}p_\textrm{g}-p_{\textrm{d,v}\theta} \\
C_{\textrm{v}\theta}q_\textrm{g}-q_{\textrm{d,v}\theta} \\
C_\textrm{pv}q_\textrm{g}-q_\textrm{d,pv}
\end{bmatrix}}_{\zeta(p_g,q_g)}
=\underbrace{\begin{bmatrix}
\widehat{G}^\textrm{c}_{\textrm{v}\theta} & \widehat{B}^\textrm{s}_{\textrm{v}\theta} & G^\mathrm{d}_{\textrm{v}\theta} \\
-\widehat{B}^\textrm{c}_{\textrm{v}\theta} & \widehat{G}^\textrm{s}_{\textrm{v}\theta} & -B^\mathrm{d}_{\textrm{v}\theta} \\
-\widehat{B}^\textrm{c}_\textrm{pv} & \widehat{G}^\textrm{s}_\textrm{pv} & -B^\mathrm{d}_\textrm{pv}
\end{bmatrix}}_{M_\textrm{ineq}}\psi(v,\varphi).
\label{eqn:pf_ineqn}
\end{equation}
This is a subset of the AC power flow equations that are necessary to define the intermediate variables.
Line flows can be represented in terms of the phase-adjusted basis functions,
\begin{equation}
\underbrace{\begin{bmatrix} s_p^\textrm{f} \\ s_q^\textrm{f} \end{bmatrix}}_{s^\textrm{f}}
=\underbrace{\begin{bmatrix}
G_\mathrm{ft} & B_\mathrm{ft} & G_\mathrm{ff}E_\mathrm{f}^T \\
-B_\mathrm{ft} & G_\mathrm{ft} & -B_\mathrm{ff}E_\mathrm{f}^T
\end{bmatrix}}_{L_\textrm{line}^\textrm{f}}\psi(v,\varphi),
\label{eqn:lineflow_f} \end{equation}

\begin{equation}
\underbrace{\begin{bmatrix} s_p^\textrm{t} \\ s_q^\textrm{t} \end{bmatrix}}_{s^\textrm{t}}
=\underbrace{\begin{bmatrix}
G_\mathrm{tf} & -B_\mathrm{tf} & G_\mathrm{tt}E_\mathrm{t}^T \\
-B_\mathrm{tf} & -G_\mathrm{tf} & -B_\mathrm{tt}E_\mathrm{t}^T
\end{bmatrix}}_{L_\textrm{line}^\textrm{t}}\psi(v,\varphi),
\label{eqn:lineflow_t} \end{equation}
where the Appendix gives the block matrices in $L_\textrm{line}^\textrm{f}$ and $L_\textrm{line}^\textrm{t}$.

\subsection{Phase-Adjusted AC Optimal Power Flow}
The AC OPF problem can be written based on the phase-adjusted AC power flow with the consideration of state and control variables. This formulation is equivalent to the classical form of the AC OPF problem without any approximation. 
The AC OPF problem identifies the operating point with minimum generation cost while respecting the operational constraints:

\begin{subequations}
\begin{equation}
\underset{x,u,\overline{s}^\textrm{f},\overline{s}^\textrm{t}}{\text{minimize}} \hskip 1em c(p_\textrm{g})=\sum_{i=1}^{n_g} c_i (p_{\textrm{g},i})
\label{eqn:cost}
\end{equation}
\begin{equation}
\begin{aligned}
\text{subject to} \hskip 2em \tau(u)+M_\textrm{eq}\psi(v,\varphi)=0
\end{aligned}\label{eqn:opf_eqn}\end{equation}
\begin{equation}\begin{aligned}
\zeta(p_g^\textrm{min},q_g^\textrm{min})\leq M_\textrm{ineq}\psi(v,\varphi)\leq \zeta(p_g^\textrm{max},q_g^\textrm{max})
\end{aligned}
\label{eqn:opf_ineqn}
\end{equation}
\begin{equation}
    \underbrace{\begin{bmatrix} E^T_\textrm{ns} & \mathbf{0} \\
    \mathbf{0} & I \\
    -E^T_\textrm{ns} & \mathbf{0} \\
    \mathbf{0} & -I \\
    \end{bmatrix}}_{A}x\leq
    \underbrace{\begin{bmatrix}
    \varphi^\textrm{max} \\ v^\textrm{max}_\textrm{pq} \\ -\varphi^\textrm{min} \\ -v^\textrm{min}_\textrm{pq}
    \end{bmatrix}}_{b^\textrm{max}},\ 
    \underbrace{\begin{bmatrix} p_{\textrm{pv}}^\textrm{min} \\ v_\textrm{g}^\textrm{min} \end{bmatrix}}_{u^\textrm{min}} \leq u\leq \underbrace{\begin{bmatrix} p_{\textrm{pv}}^\textrm{max} \\ v_\textrm{g}^\textrm{max} \end{bmatrix}}_{u^\textrm{max}}
    \label{eqn:xu_lim}
\end{equation}
\begin{equation} 
\begin{aligned}
\lvert L_\textrm{line}^\textrm{f} \psi(x,u)\rvert &\leq \overline{s}^\textrm{f},\ \lvert L_\textrm{line}^\textrm{t} \psi(x,u)\rvert \leq \overline{s}^\textrm{t} \\
\left(\overline{s}_p^{\textrm{f}}\right)^2+\left(\overline{s}_q^{\textrm{f}}\right)^2&\leq \left(s^\textrm{max}\right)^2,\
\left(\overline{s}_p^{\textrm{t}}\right)^2+\left(\overline{s}_q^{\textrm{t}}\right)^2\leq \left(s^\textrm{max}\right)^2
\end{aligned} \label{eqn:opf_linelim} 
\end{equation}
\label{eqn:OPF}\end{subequations}
The cost function of each generator $i$, $c_i(p_{g,i})$, is assumed to be monotonically increasing with respect to the active power generation. Equation \eqref{eqn:opf_eqn} contains the subset of power flow equations which relate the control and state variables, and the matrix $M_\textrm{eq}$ is defined in \eqref{eqn:pf_eqn}. Equation \eqref{eqn:opf_ineqn} imposes constraints on the intermediate variables (the active power on the V$\theta$ bus and reactive power on generator buses). 
The matrix $M_\textrm{ineq}$ and function $\zeta(p_g,q_g)$ are defined in \eqref{eqn:pf_ineqn}.
Equation \eqref{eqn:xu_lim} imposes the voltage magnitude limits, active power limits for PV buses, and angle limits. Equation \eqref{eqn:opf_linelim} imposes the line flow limits with $L^\textrm{f}_\textrm{line}$ and $L^\textrm{t}_\textrm{line}$ defined in \eqref{eqn:lineflow_f} and \eqref{eqn:lineflow_t}.
The matrix $E_\textrm{ns}\in\mathbb{R}^{(n_b-1)\times n_l}$ is a submatrix of $E$ that selects the rows corresponding to the non-slack buses. %Note that $E^T\theta=E_\textrm{ns}^T\theta_\textrm{ns}$ given $\theta_{\textrm{v}\theta}=0$. 
\section{Optimal Power Flow with Convex Restriction}
\label{sec:restriction}

In this section, we summarize the procedure of obtaining a \emph{convex restriction} for the AC OPF problem. A convex restriction provides a convex condition on the control variables $u$ such that there exists state variables $x$ that satisfy both the AC power flow equations in \eqref{eqn:DACPA} and the operational constraints in \eqref{eqn:OPconstr}. A sufficient convex condition for AC power flow feasibility was developed in \cite{Lee2018}, and we extend its application to solve the full OPF problem including line flow limits.

\subsection{Quadratic Convex Restriction of the Feasible Space}
\subsubsection{Power Flow Constraints in Fixed Point Form}
The convex restriction is constructed around the known, feasible base point $(x_0,u_0)$, which is assumed to have a non-singular power flow Jacobian with respect to the state variables.\footnote{If the power flow Jacobian is singular, the system is operating at the nose of PV curve where the solution to power flow equations can disappear by an arbitrarily small perturbation in the power injections.} 
Consider the power flow equations in \eqref{eqn:pf_eqn} as finding the zeros of  $f(x,u)=\tau(u)+M_\textrm{eq}\psi(v,\varphi)$. Let us denote the Jacobian with respect to $x$ as $J_{f,0}=\nabla_x f|_{(x_0,u_0)}=M_\textrm{eq}J_{\psi,0}$ where $J_{\psi,0}=\nabla_x \psi(v,E^T\theta)|_{(v_0,\varphi_0)}$. Then, we can write the power flow equations in the following fixed-point form
\begin{equation}\begin{aligned}
    x&=-J_{f,0}^{-1}(f(x,u)-J_{f,0}x)\\
    &=-J_{f,0}^{-1}\left(M_{eq}g(x,u)+\tau(u)\right),
\end{aligned}\label{eq:fixed-point-form}
\end{equation}
where $g(x,u)$ represents the residual of the basis functions,
\begin{equation}
    g(x,u) = \psi(v,\varphi) - J_{\psi,0}x.
\end{equation}
Note that \eqref{eq:fixed-point-form} corresponds to a single iteration of the Newton-Raphson procedure, which is commonly used to solve the power flow equations.

\subsubsection{Sufficient Condition for Existence of $x$}
The derivation of the sufficient condition for AC power flow solvability relies on Brouwer's Fixed Point Theorem.
\begin{theorem} (Brouwer's Fixed Point Theorem \cite{brouwer1911abbildung}) 
Let $\mathcal{P}\subseteq\mathbb{R}^n$ be a nonempty compact convex set and $F\!:\!\mathcal{P}\!\rightarrow\!\mathcal{P}$ be a continuous mapping. Then there exists some $x\in\mathcal{P}$ such that $F(x)=x$.
\end{theorem}
In our approach, the map $F$ corresponds to the power flow equations \eqref{eq:fixed-point-form}. We define the self-mapping set $\mathcal{P}$ as 
% \begin{equation}
% \begin{aligned}
%     \mathcal{P}(b)\!=\!\{x\mid \underline{\varphi}\!\leq\!\varphi\!\leq\!\overline{\varphi},\ \underline{v}_\textrm{pq}\!\leq\! v_\textrm{pq}\!\leq\!\overline{v}_\textrm{pq}\} 
%     \!=\!\{x\mid Ax\!\leq\! b\},
% \end{aligned}
% \end{equation}
\begin{equation}
\begin{aligned}
    \mathcal{P}(b)&=\{x\mid \underline{\varphi}\leq\varphi\leq\overline{\varphi},\ \underline{v}_\textrm{pq}\leq v_\textrm{pq}\leq\overline{v}_\textrm{pq}\} \\
    &=\{x\mid Ax\leq b\},
\end{aligned}
\end{equation}
where the matrix $A$
%$A\in\mathbb{R}^{(2n_{pq}+2n_l)\times(n_b-1+n_{pq})}$ 
is defined in \eqref{eqn:xu_lim} and the bound $b$
%$b\in\mathbb{R}^{(2n_{pq}+2n_l)}$ 
is
\begin{equation}
b=\begin{bmatrix}\overline{\varphi}^T & \overline{v}_\textrm{pq}^T & -\underline{\varphi}^T & -\underline{v}_\textrm{pq}^T \end{bmatrix}^T.
\label{eqn:b_def}
\end{equation}
The polytope $\mathcal{P}(b)$ is a closed and compact set parametrized by $b$, which provides the upper and lower bounds on the state variables. The bounds $b$ are \emph{not} the same as the limits provided in %\eqref{eqn:vmaglim}, 
\eqref{eqn:anglelim}, but are decision variables. Then Brouwer's fixed point condition is equivalent to the existence of $b\in\mathbb{R}^{(2n_{pq}+2n_l)}$ such that
\begin{equation}
    \max_{x\in\mathcal{P}(b)} Kg(x,u)-AJ_{f,0}^{-1}\tau(u) \leq b,
    %\ \ \textrm{where} \ \ K=-AJ_{f,0}^{-1}M_{eq}.
    \label{eqn:condition}
\end{equation}
where $K=-AJ_{f,0}^{-1}M_{eq}$.

\subsubsection{Concave envelopes and bounds for $g(x,u)$ and $\psi(x,u)$}
A \emph{concave envelope} of a function $g(x,u)$ is given by a concave under-estimator 
$\underline{g}_k(x,u)$ and a convex over-estimator $\overline{g}_k(x,u)$, such that 
\begin{equation}
    \underline{g}_k(x,u) \leq g_k(x,u) \leq \overline{g}_k(x,u).
\end{equation}
Given this concave envelope, the bound on $g_k$ over the domain $\mathcal{P}(b)$ is
\begin{equation}
\label{eq:gbar}
\begin{aligned}
\overline{g}_{\mathcal{P},k}(u,b)&\geq\max_{x\in\mathcal{P}(b)}\overline{g}_k(x,u)=\max_{x\in\partial\mathcal{P}_k(b)}\overline{g}_k(x,u), \\
\underline{g}_{\mathcal{P},k}(u,b)&\leq\min_{x\in\mathcal{P}(b)}\underline{g}_k(x,u)=\min_{x\in\partial\mathcal{P}_k(b)}\underline{g}_k(x,u),
\end{aligned}
\end{equation}
where $\partial\mathcal{P}_k(b)$ is the set of vertices in polytope $\mathcal{P}(b)$ that are involved in function $g_k(x,u)$.
For the second equality, we exploit the fact that since the envelopes are concave for the minimization problem and convex for the maximization problem in~\eqref{eq:gbar}, the extreme values $\overline{g}_{\mathcal{P},k}(u,b)$, $\underline{g}_{\mathcal{P},k}(u,b)$ will occur at one of the vertices in $\partial\mathcal{P}(b)$. 
Hence, we can ensure that the max/min inequalities holds over the polytope by requiring all vertices to satisfy the above inequalities.
Fig.~\ref{fig:concave_env} illustrates the concave envelope and the bounds over the polytope $\mathcal{P}(b)$ for an example function.

\begin{figure}%[!htbp]
	\centering
 	\includegraphics[width=2.4in]{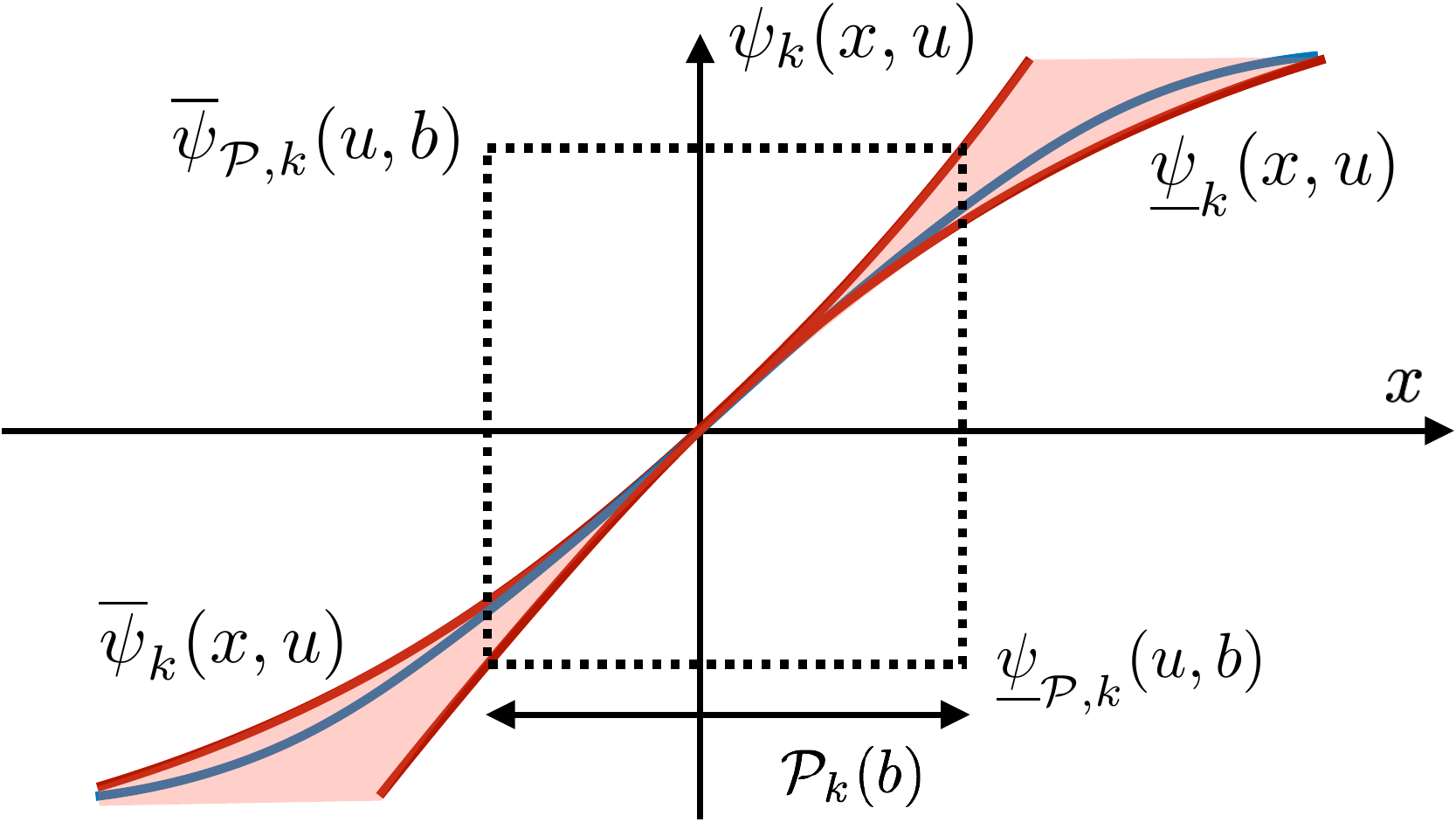}
	\caption{Illustration of the concave envelope (in red) for the function $\psi(x,u)$ (in blue) and the corresponding bounds on $\psi(x,u)$ over the interval $\mathcal{P}_k(b)$.}
	\label{fig:concave_env}
\end{figure}

In the power flow equations, the functions $g(x,u)$ can be expressed as a combination of bilinear, cosine, and sine functions. The concave envelopes for these functions from \cite{Lee2018} are provided below. 
The envelopes of a bilinear function are
\begin{displaymath}
\begin{aligned}
\langle xy \rangle^{Q} &\geq-\frac{1}{4}[(x-x_0)-(y-y_0)]^2+x_0y+xy_0-x_0y_0, \\
\langle xy \rangle^{Q}&\leq\frac{1}{4}[(x-x_0)+(y-y_0)]^2+x_0y+xy_0-x_0y_0.
\end{aligned}
\end{displaymath}
For trigonometric functions, we exploit the angle difference limits with the phase-adjusted power flow formulation to construct a tight concave envelope.
Assuming $\varphi^\textrm{max} \in \left[0,\,\pi\right]$ and $\varphi^\textrm{min} \in \left[-\pi,\,0\right]$, the concave envelopes for the sine and cosine functions are
\begin{displaymath}
\begin{aligned}
\langle\sin\tilde{\varphi}\rangle^{S}&\geq\tilde{\varphi}+\left(\frac{\sin\tilde{\varphi}^\textrm{max}-\tilde{\varphi}^\textrm{max}}{(\tilde{\varphi}^\textrm{max})^2}\right)\tilde{\varphi}^2, \ \tilde{\varphi}<\tilde{\varphi}^\textrm{max}, \\
\langle\sin\tilde{\varphi}\rangle^{S}&\leq \tilde{\varphi}+\left(\frac{\sin\tilde{\varphi}^\textrm{min}-\tilde{\varphi}^\textrm{min}}{(\tilde{\varphi}^\textrm{min})^2}\right)\tilde{\varphi}^2, \ \tilde{\varphi}>\tilde{\varphi}^\textrm{min},\\
\langle\cos\tilde{\varphi}\rangle^{C}&\geq1-\frac{1}{2}\tilde{\varphi}^2, \ \
\langle\cos\tilde{\varphi}\rangle^{C}\leq1.
\end{aligned}
\end{displaymath}

The upper bounds on $g(x,u)$ over $\mathcal{P}(b)$ are defined as
\begin{displaymath}
\begin{aligned}
\overline{g}_{\mathcal{P},l}^C(u,b)&\geq\max_{x_l\in\mathcal{X}_l} \langle\langle v^\mathrm{f}_lv^\mathrm{t}_l\rangle^Q \langle\cos{\tilde{\varphi}_l}\rangle^C\rangle^Q-v^\mathrm{f,pq}_{0,l}v^\mathrm{t}_l-v^\mathrm{f}_lv^\mathrm{t,pq}_{0,l} \\
\overline{g}_{\mathcal{P},l}^\mathrm{S}(u,b)&\geq\max_{x_l\in\mathcal{X}_l} \langle\langle v^\mathrm{f}_lv^\mathrm{t}_l\rangle^Q \langle\sin{\tilde{\varphi}_l}\rangle^S\rangle^Q-v^\mathrm{f}_{0,l}v^\mathrm{t}_{0,l}\tilde{\varphi}_l \\
\overline{g}_{\mathcal{P},k}^Q(u,b)&\geq\max_{v_k\in\{\overline{v}_k,\underline{v}_k\}}\langle v_kv_k\rangle^Q-2v^\mathrm{pq}_{0,k}v_k,
\end{aligned}
\end{displaymath}
where the constant $v^\mathrm{pq}_{0,k}$ is set to $v_{k,0}$ if bus $k$ is a PQ bus and 0 if it is non-PQ bus. Similarly, $v^\mathrm{f,pq}_{0,l}$ and $v^\mathrm{t,pq}_{0,l}$ are set to $v^\mathrm{f}_{0,l}$ and $v^\mathrm{t}_{0,l}$ if they are PQ buses and 0 if they are \mbox{non-PQ} buses. The vertices are defined by $x_l=(v_l^\textrm{f},\,v_l^\textrm{t},\,\tilde{\varphi}_l)$ and $\mathcal{X}_l=\{(v_l^\textrm{f},v_l^\textrm{t},\tilde{\varphi}_l)\mid v_l^\textrm{f}\in\{\underline{v}_l^\textrm{f},\,\overline{v}_l^\textrm{f}\},\,v_l^\textrm{t}\in\{\underline{v}_l^\textrm{t},\,\overline{v}_l^\textrm{t}\},\,\tilde{\varphi}_l\in\{\underline{\varphi}_k-\varphi_{0,k},\,\overline{\varphi}_k-\varphi_{0,k}\}\}$. 

Note that the number of vertices that need to be checked for each line is constant, i.e., the cardinality of $\mathcal{X}_l$ is $2^3$ regardless of the size of the system. Similarly, $\underline{g}_{\mathcal{P},k}(u,b)$ are defined by replacing maximization with minimization and changing the direction of inequality sign, and $\overline{\psi}_{\mathcal{P},k}(u,b)$ and $\underline{\psi}_{\mathcal{P},k}(u,b)$ are defined by replacing the function $g$ by $\psi$.

\subsubsection{Convex Restriction of the OPF Feasible Space}
The above upper and lower bounds on $g(x,u)$ over the region $\mathcal{P}(b)$ allow us to guarantee that condition \eqref{eqn:condition} for power flow feasibility holds.
Similarly, the bounds on $\psi(x,u)$ are used to ensure satisfaction of the inequality constraints \eqref{eqn:opf_ineqn}--\eqref{eqn:opf_linelim}.
The resulting convex restriction represents a convex inner approximation of the feasible space in the OPF problem.
This is proven by the following Theorem from \cite{Lee2018}, which we extend to include line flow limits.
\begin{theorem}{(Convex Restriction of Power Flow Feasibility Constraints)}
Given the operating point $u=(p_{\textrm{pv}},\, v_\textrm{g})$, there exists a solution for the state $x=(\theta_\textrm{ns},v_\textrm{pq})$ that satisfies the AC OPF constraints \eqref{eqn:DACPA}, \eqref{eqn:OPconstr} if there exist
$b=(\overline{\varphi},\,\overline{v}_\textrm{pq},\,-\underline{\varphi},\,-\underline{v}_\textrm{pq})$ and $(\overline{s}^\textrm{f},\overline{s}^\textrm{t})$ such that the following conditions hold:
\begin{equation}\begin{aligned}
-AJ_{f,0}^{-1}\tau(u)+K^+ \overline{g}_\mathcal{P}(u,b)+K^- \underline{g}_\mathcal{P}(u,b) &\leq b,\end{aligned}\label{eqn:conv_restr_eq}\end{equation}
\vspace{-10pt}
\begin{equation}\begin{aligned}
M_\textrm{ineq}^+\overline{\psi}_\mathcal{P}(u,b)+M_\textrm{ineq}^-\underline{\psi}_\mathcal{P}(u,b)&\leq \zeta(p_g^\textrm{max},q_g^\textrm{max}), \\
M_\textrm{ineq}^-\overline{\psi}_\mathcal{P}(u,b)+M_\textrm{ineq}^+\underline{\psi}_\mathcal{P}(u,b)&\geq\zeta(p_g^\textrm{min},q_g^\textrm{min}), \\
b\leq b^\textrm{max}, u^\textrm{min}\leq u&\leq u^\textrm{max}, \\
L_\textrm{line}^{\textrm{k},+}\overline{\psi}_\mathcal{P}(u,b)+L_\textrm{line}^{\textrm{k},-}\underline{\psi}_\mathcal{P}(u,b)&\leq \overline{s}^\textrm{k},\ \ \textrm{k}\in\{\textrm{f},\textrm{t}\},\\
-L_\textrm{line}^{\textrm{k},-}\overline{\psi}_\mathcal{P}(u,b)-L_\textrm{line}^{\textrm{k},+}\underline{\psi}_\mathcal{P}(u,b)&\leq \overline{s}^\textrm{k},\ \ \textrm{k}\in\{\textrm{f},\textrm{t}\},\\
&\hskip -8.2em \left(\overline{s}_p^{\textrm{k}}\right)^2+\left(\overline{s}_q^{\textrm{k}}\right)^2\leq \left(s^\textrm{max}\right)^2,\ \ \textrm{k}\in\{\textrm{f},\textrm{t}\},
\end{aligned} \label{eqn:conv_restr_ineq} \end{equation}
where $K=-AJ_{f,0}^{-1}M_\textrm{eq}$, and $\Lambda^+_{ij}=\max\{\Lambda_{ij},0\}$ and $\Lambda^-_{ij}=\min\{\Lambda_{ij},0\}$ for an arbitrary matrix $\Lambda$.
\label{thm_feasibility}
\end{theorem}
\begin{proof}
Condition \eqref{eqn:conv_restr_eq} is sufficient to satisfy condition \eqref{eqn:condition} in Brouwer's Fixed Point Theorem:
\begin{equation*}
\begin{aligned}
    &\max_{x\in\mathcal{P}(b)} Kg(x,u)-AJ_{f,0}^{-1}\tau(u) \\ 
    &\hskip 3em \leq K^+ \overline{g}_\mathcal{P}(u,b)+K^- \underline{g}_\mathcal{P}(u,b)-AJ_{f,0}^{-1}\tau(u)\leq b.
\end{aligned}
\end{equation*}
Then for all $x\in\mathcal{P}(b), \ -J_{f,0}^{-1}M_\textrm{eq}g(x,u)\in\mathcal{P}(b)$. By applying Brouwer's Fixed Point Theorem to \eqref{eq:fixed-point-form}, there exists a solution $x\in\mathcal{P}(b)$. Further, \eqref{eqn:conv_restr_ineq} ensures the operational constraints are satisfied for all $x\in\mathcal{P}(b)$. A more detailed proof is in \cite{Lee2018}.
\end{proof}
Note that the number of inequality constraints in \eqref{eqn:conv_restr_eq} and \eqref{eqn:conv_restr_ineq} is linearly proportional to the system size. Further, the convex restriction can be written analytically, and the only necessary computation is the inversion of the Jacobian matrix at the base operating point.

\subsection{Optimal Power Flow with Quadratic Convex Restriction}
We obtain an inner convex approximation of the AC OPF problem's feasible space by replacing the original AC OPF constraints \eqref{eqn:DACPA}, \eqref{eqn:OPconstr} with the convex restriction \eqref{eqn:conv_restr_eq}, \eqref{eqn:conv_restr_ineq}. The objective function requires further consideration. 

\subsubsection{Objective Function}
The objective is a function of the active power output from each generator ($p_{g,i}$). Since the active power generation at the slack (V$\theta$) bus is an implicit state variable, it is replaced by its over-estimator.
Given that the objective function is monotonically increasing with respect to the active power generation, the objective can be over-estimated by
\begin{equation}
    \overline{c}(u,b)=c_{\textrm{v}\theta}(\overline{p}_{\textrm{g,v}\theta})+\sum_{i=1}^{n_\textrm{pv}} c_{\textrm{pv},i} (p_{\textrm{pv},i})
    \label{eqn:cvxrs_cost}
\end{equation}
where $\overline{p}_{\textrm{g,v}\theta}$ is an over-estimator on the active power generated at the slack bus. This over-estimator is constrained by
\begin{equation}
C_{\textrm{v}\theta}\,\overline{p}_{\textrm{g,v}\theta}-p_{\textrm{d,v}\theta}\geq M_{\textrm{v}\theta}^+ \,\overline{\psi}_\mathcal{P}(u,b) + M_{\textrm{v}\theta}^- \, \underline{\psi}_\mathcal{P}(u,b),
\label{eqn:p_slack}
\end{equation}
where $M_{\textrm{v}\theta}\in\mathbb{R}^{1\times(2n_l+n_b)}$ is the row of $M_\textrm{ineq}$ that corresponds to the active power generation limit at the V$\theta$ bus.

\subsubsection{Optimal Power Flow with Convex Restriction}\label{subsubsec:cvxrs_opf}
The AC~OPF problem can be solved by minimizing the upper bound on the cost function subject to the convex restriction's conditions. The decision variables are the control variables $u$ and the state variable bounds $b$ as well as the intermediate variables that represent upper bounds on the line flows $(\overline{s}^\textrm{f},\overline{s}^\textrm{t})$ and slack bus active power generation $\overline{p}_{\textrm{g,v}\theta}$, respectively. The resulting optimization problem is

\begin{displaymath}
	\begin{aligned}
		\underset{u, b, \overline{s}^\textrm{f},\overline{s}^\textrm{t}, \overline{p}_{\textrm{g,v}\theta}}{\text{minimize}} \hskip 1em & \eqref{eqn:cvxrs_cost} &&\textrm{ objective function} \\
		\text{subject to} \hskip 1em &  \eqref{eqn:conv_restr_eq},\,  \eqref{eqn:conv_restr_ineq},\, \eqref{eqn:p_slack} &&\textrm{ convex restriction}.
	\end{aligned}
	\label{eqn:cvxopf}
\end{displaymath}

\begin{remark}
The solution obtained via convex restriction ($p_\textrm{g}^\textrm{cvxrs}$) is lower bounded by the globally optimal solution of the original AC OPF problem ($p_\textrm{g}^*$) and is upper bounded by the objective value at the base point ($p_{\textrm{g},0}$):
\begin{displaymath}
c(p_\textrm{g}^*)\leq c(p_\textrm{g}^\textrm{cvxrs})\leq c(p_{\textrm{g},0}).
\end{displaymath}
\label{remark:obj_bound}\end{remark}
\begin{remark}
The OPF with convex restriction (i.e., \eqref{eqn:conv_restr_eq}--\eqref{eqn:p_slack}) is a convex quadratically constrained quadratic program (convex QCQP) that can be solved with commercial solvers such as Mosek, CPLEX, and Gurobi. The number of quadratic constraints is bounded by $30n_l+4n_b+4n_g$. 
\label{remark:scalability}\end{remark}

\cref{remark:obj_bound} states that the solution of OPF with convex restriction has reduced or equal objective value relative to the base point. 
\cref{remark:scalability} indicates that the size of the resulting convex optimization problem increases linearly with the system size.

\section{Feasible Path Optimal Power Flow}
\label{sec:path}
We present an iterative algorithm to solve OPF with the convex restriction, while guaranteeing the existence of a feasible path to the new operating point.

\subsection{Definition of the Feasible Path}
The motivation for studying the feasible path is to bring the system from the current operating point to the desired operating point while guaranteeing steady-state stability, i.e., a trajectory which satisfies the AC power flow equations as well as the operational constraints.
This leads to the following definition of a feasible path.
\begin{definition}
A \textit{feasible path} between two control set points $u^{(0)}$ and $u^{(N)}$ is a set of control variables that forms a continuous trajectory connecting the two set points such that there exists state variables $x$ that satisfy the AC OPF constraints in \eqref{eqn:DACPA}, \eqref{eqn:OPconstr} at every point along the trajectory.
\end{definition}
In particular, the feasible path will be described by a sequence of control actions $u^{(k)}, i=0,\ldots,N$ where
\begin{displaymath}
    \mathcal{U}^\textrm{path}=\{\alpha u^{(k)}+(1-\alpha)u^{(k+1)}\mid \alpha\in[0,1],\, k=0,\ldots,N-1\}.
\end{displaymath} 
\subsection{Feasible Path Identification Algorithm}
The convex restriction provides an \emph{inner approximation} of the power flow feasibility set that is a convex set. By the definition of a convex set, all points on the line connecting two operating points $u^{(k)}$ and $u^{(k+1)}$ within the convex restriction are guaranteed to be feasible. That is, for $\alpha\in [0,1]$,
\begin{equation}
    \alpha u^{(k)}+(1-\alpha)u^{(k+1)}\in\mathcal{U}^\textrm{cvxrstr}_{(k)}.
\label{eqn:convexity_def}\end{equation}
Here, $\mathcal{U}^\textrm{cvxrstr}_{(k)}$ denotes the convex restriction \eqref{eqn:conv_restr_eq}, \eqref{eqn:conv_restr_ineq}, constructed with the base point at $u^{(k)}$. 
By leveraging this property of convexity, we propose to use \emph{sequential convex restrictions} to identify a feasible path.
The algorithm based on sequential convex restrictions is described in \cref{alg:SCRS}. Given a current set of control variables $u^{(k)}$, each iteration of the algorithm (i) solves the power flow equations to obtain the base point $(x^{(k)}, u^{(k)})$, (ii) constructs the convex restriction at this base point, and (iii) solves a convex restriction OPF to obtain a new set of control variables $u^{(k+1)}$. The output of the algorithm is a sequence of control set points $u^{(k)}, k=0,\ldots,N$, that forms a piece-wise linear feasible path between the initial operating point $u^{(0)}$ and $u^{(N)}$.

\begin{algorithm}%[!htbp]
 \begin{algorithmic}[1]
  \STATE \textit{Initialize}: Set $u^{(0)}$ and $x^{(0)}$ to the initial operating point.
  \WHILE {$\lVert u^{(k+1)}-u^{(k)}\rVert_2>\varepsilon$} 
  \STATE Solve power flow given $u^{(k)}$ to obtain $x^{(k)}$.
  \STATE Set $x_0=x^{(k)}$ and compute the power flow Jacobian at the base point ($J_{f,0}$).
  \STATE Construct Convex Restriction ($\mathcal{U}_{(k)}^\textrm{cvxrstr}$) with \eqref{eqn:conv_restr_eq}, \eqref{eqn:conv_restr_ineq}.
  \STATE Solve
  \vspace{-5pt}
  \begin{equation}
  u^{(k+1)}=\argmin_{u\in\mathcal{U}_{(k)}^\textrm{cvxrstr}} f_0(u,b).
  \label{eqn:alg_obj}\end{equation}
  \vspace{-10pt}
  \STATE $k:=k+1$.
  \ENDWHILE
 \RETURN $u^{(1)},\ldots,u^{(N)}$.
 \end{algorithmic}
 \caption{Feasible Path Identification Algorithm with Sequential Convex Restriction}
 \label{alg:SCRS}
 \end{algorithm}

Due to the property of convexity explained in \eqref{eqn:convexity_def}, the output of the algorithm provides a feasible path for the OPF problem in \eqref{eqn:OPF}.
\begin{corollary}
The piece-wise linear path from $u^{(0)}$ to $u^{(N)}$,
\begin{displaymath}
    \mathcal{U}^\textrm{path}=\{\alpha u^{(k)}+(1-\alpha)u^{(k+1)}\mid \alpha\in[0,1],\, k=0,\ldots,N-1\},
\end{displaymath}
provided by Algorithm \ref{alg:SCRS} is a feasible path with respect to the OPF constraints in \eqref{eqn:DACPA} and \eqref{eqn:OPconstr}.
\end{corollary}

The objective function of the algorithm, $f_0(u,b)$, can be designed to either directly reduce the generation cost or aim to get close to a known desirable operating point.

\subsection{Operational Scenarios for Feasible Path OPF}
Depending on the problem setting, we might want to consider different objective functions $f_0$ in \eqref{eqn:alg_obj}. We provide two examples.

\subsubsection{Optimal Power Flow with Feasible Path Guarantees}
Given the current, sub-optimal operating point $(x_0, u_0)$, we want to find a lower-cost operating point $(x^*, u^*)$ while providing a certifiably feasible path between the two points.
In this formulation, the OPF problem is directly solved by setting the objective to be the cost of generation:
\begin{equation}
    f_0(u,b)=\bar{c}(u,b),
    \label{eqn:pathguarantees}
\end{equation}
where $\bar{c}(u,b)$ is defined in \eqref{eqn:cvxrs_cost}. 
The algorithm solves the OPF with convex restriction as discussed in \cref{subsubsec:cvxrs_opf} and iterates by setting the solution to the base point and repeating the process of solving OPF with convex restriction.

\subsubsection{Feasible Path Identification for Known Operating Points}
In an alternative scenario, we are provided a known, desired operating point $(p_\textrm{pv}^*,v_\textrm{g}^*)$ and seek a sequence of feasible control actions which bring the system towards the desired point. The objective function here  minimizes the square of the Euclidean distance from the desired generation set point,
\begin{equation}
    f_0(u,b)=\lambda\lVert p_\textrm{pv}-p_\textrm{pv}^* \rVert_2^2+\lVert v_\textrm{g}-v_\textrm{g}^* \rVert_2^2,
    \label{eqn:FPI_obj}
\end{equation}
where $\lambda$ is a relative weighting of the differences in generator power injections and voltage magnitudes. The convergence of the algorithm depends on the weight $\lambda$, which will be investigated further in the numerical studies section.

\subsection{Convergence of the Feasible Path OPF}
The sequential convex restriction may not always converge to the optimal solution. We provide a few scenarios in which the algorithm may not arrive at the desired operating point.
\begin{itemize}
    \item If the initial and the optimal operating points belong to separate, disconnected regions of the feasible space, there is no feasible path between the two points. The algorithm's final point will reside in the region with the initial point. 
    \item The algorithm could converge to a point at the nonconvex boundary of the feasible set where all cost-descending directions are infeasible.
    \item The dispatch point could get too close to the voltage collapse point. As it gets close to the voltage collapse point, the power flow Jacobian, $J_{f,0}$, will become singular.
\end{itemize}

The next section provides quantitative experiments to show the convergence of the algorithm on standard IEEE test cases.

\section{Numerical Studies}
\label{sec:experiments}
This section computationally demonstrates our algorithm.
Two illustrative examples are presented to visualize how the algorithm finds a feasible path to a desired point and to study the convergence characteristics. Extensive numerical studies show how the algorithm improves the initial operating point.

\subsection{Implementation}
The studies were conducted using the test cases from PGLib v19.01 \cite{pglib} with sizes up to the 588 buses. Computations were done using a laptop with a 3.3 GHz Intel Core i7 processor and 16 GB of RAM. The code implementation uses JuMP/Julia~\cite{jump}. MOSEK was used to solve the convex QCQPs associated with the convex restrictions. M{\sc atpower}'s interior point method was used to solve the AC OPF problems to obtain desired operating points~\cite{zimmerman11}. \cref{alg:SCRS} was implemented with $\varepsilon=0.01$, where the power flow in each iteration (step~3) was solved using the Newton-Raphson method.

\subsection{Illustrative Example using a Two-Bus System}

This section presents a two-bus system composed of a controllable load connected to the slack bus via a line with impedance $z=j1$. In contrast to the normal practice for OPF problems, we model the active and reactive power at the load bus as control variables for the sake of illustration. Considering a constant voltage magnitude of 1 p.u. at the slack bus, the power flow equations are
\begin{equation}\begin{aligned}
p&=v\sin\varphi \\
q&=-v\cos\varphi+v^2.
\end{aligned}\label{eqn:2bus}\end{equation}
The control variables are the active and reactive power at the load bus ($p\in\mathbb{R}$ and $q\in\mathbb{R}$), and the state variables are voltage magnitude and the phase angle at the load bus ($v$ and $\varphi$). In addition to the power flow equations, the voltage magnitude is constrained as $0.9\leq v\leq 1.1$.

\begin{figure}[!t]
	\centering
	\includegraphics[width=2.6in]{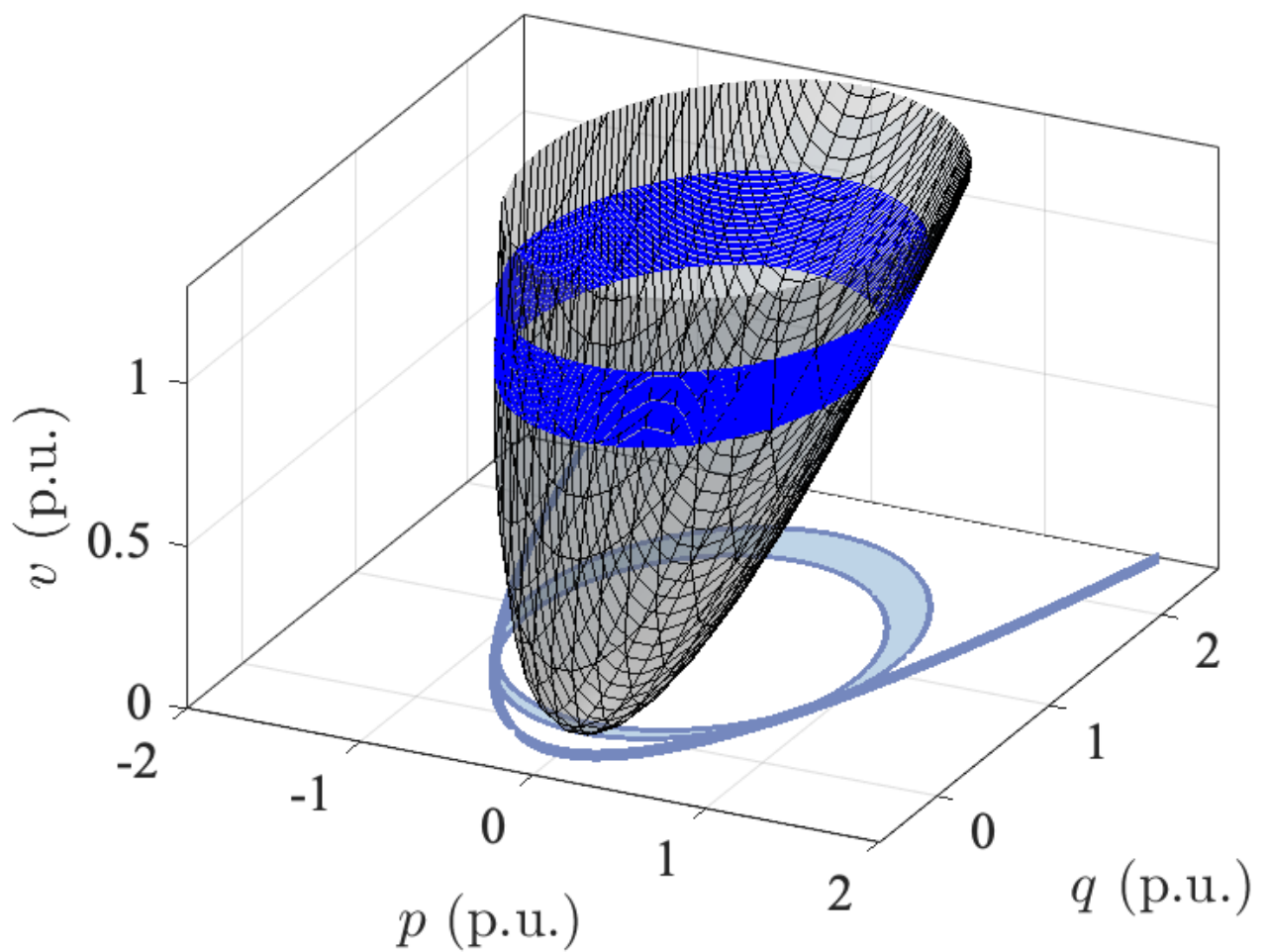}
	\caption{The nonlinear manifold created by the AC power flow equations for the two-bus system. The set of states that satisfy the voltage magnitude limits is colored in blue. The projection of the manifold and the feasible region onto the space of active and reactive power are shown on the $p$ and $q$ axes.}
	\label{fig:twobus_3d}
\end{figure}

\begin{figure}[!t]
	\centering
	\includegraphics[width=2.8in]{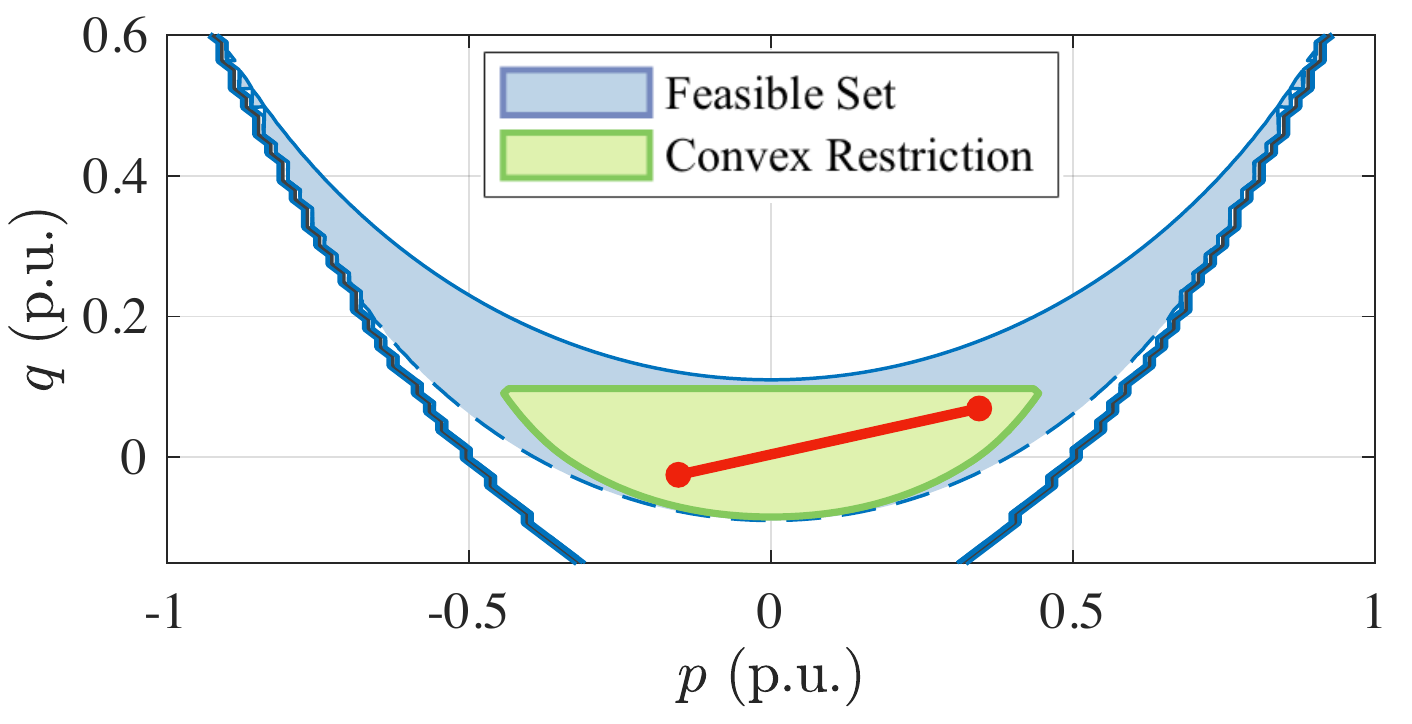}
	\caption{The feasible set and convex restriction for the control variables in the two-bus system. Any piece-wise linear path between two points in the convex restriction provides a feasible path. One such path is shown in red.}
	\label{fig:twobus_2d}
\end{figure}

\begin{figure*}[!h]
	\centering
	\includegraphics[width=5.5in]{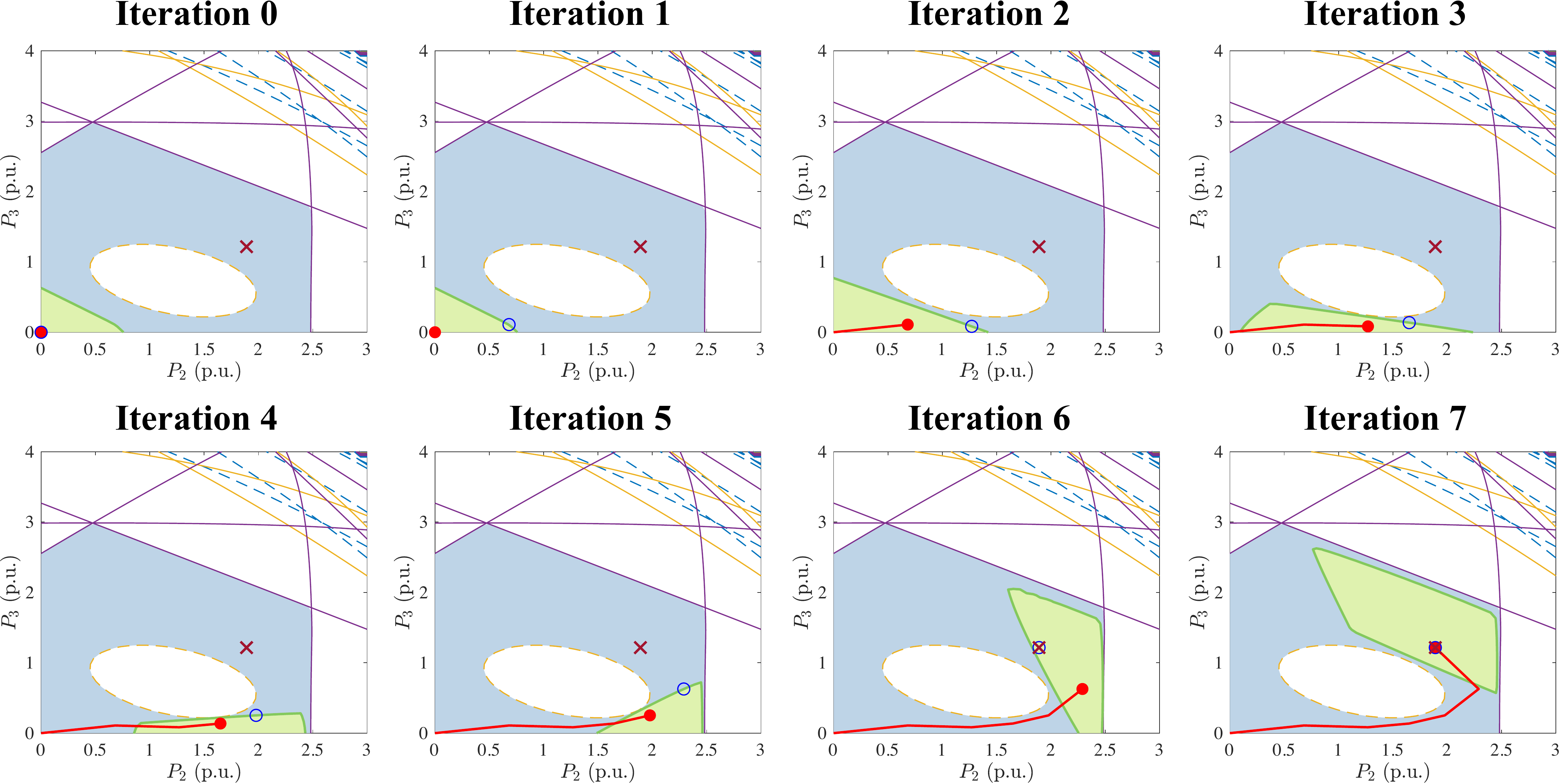}
	\caption{Illustration of sequential convex restriction in the 9-bus system. The feasibility space is shown in blue, and the convex restriction is shown in green. The desired operating point is marked by the ${\color{purple}\times}$ symbol. The red point ${\color{red}\bullet}$ indicates the base operating point where the convex restriction was constructed, and the blue circle ${\color{blue}\circ}$ shows the optimal operating point within the convex restriction. The red line shows the feasible path.}
	\label{fig_FPP_9bus}
	\vspace*{-1em}
\end{figure*}

Figure \ref{fig:twobus_3d} shows the manifold created by the power flow equations and the manifold's projection onto the space of control variables. Any $p$, $q$, and $v$ on the manifold has a phase angle $\varphi$ that satisfies the power flow equation in \eqref{eqn:2bus}. 
If the voltage magnitude limits are not present, the solvability condition over the active and reactive power at the load bus is known to be $p^2-q\leq1/4$. 

This condition corresponds to the solvability boundary given by the thick blue line in Figure~\ref{fig:twobus_2d}. This figure shows the projection of the manifold onto the control variables, focusing on the region with low reactive power consumption. With the consideration of voltage magnitude limits visualized by the thin solid and dashed blue lines, the projection is represented in blue. The green region shows the convex restriction condition constructed around $p=0$ and $q=0$. Due to the properties of convex sets, the linear paths between any two points within the convex restriction are guaranteed to be feasible. The red line in Figure \ref{fig:twobus_2d} gives an example of one such feasible path.

\subsection{Illustrative Example: Feasible Path for a 9-Bus System}
For the second illustrative example, we consider the 9-bus system from~\cite{chow1982time}. In this experiment, the voltage magnitudes at the generators were fixed at 1 p.u. and the generators' reactive power limits were reduced from 300 MVAr to 100 MVAr. Figure \ref{fig_FPP_9bus} shows the sequential convex restriction applied to the 9-bus system, with the changes in the control variable setpoints plotted as the algorithm progresses. The sequential convex restriction is minimizing the square of the distance to the desired operating point (marked by the ${\color{purple}\times}$ symbol) at each iteration, using the objective function \eqref{eqn:FPI_obj} with $\lambda=1$. The algorithm converges to the desired operating point in 7 iterations. The figure shows that the piece-wise linear path goes around the infeasible operating region (plotted in white) and arrives at the desired operating point without violating any OPF constraints.

\begin{figure}[!b]
	\centering
	\includegraphics[width=3in]{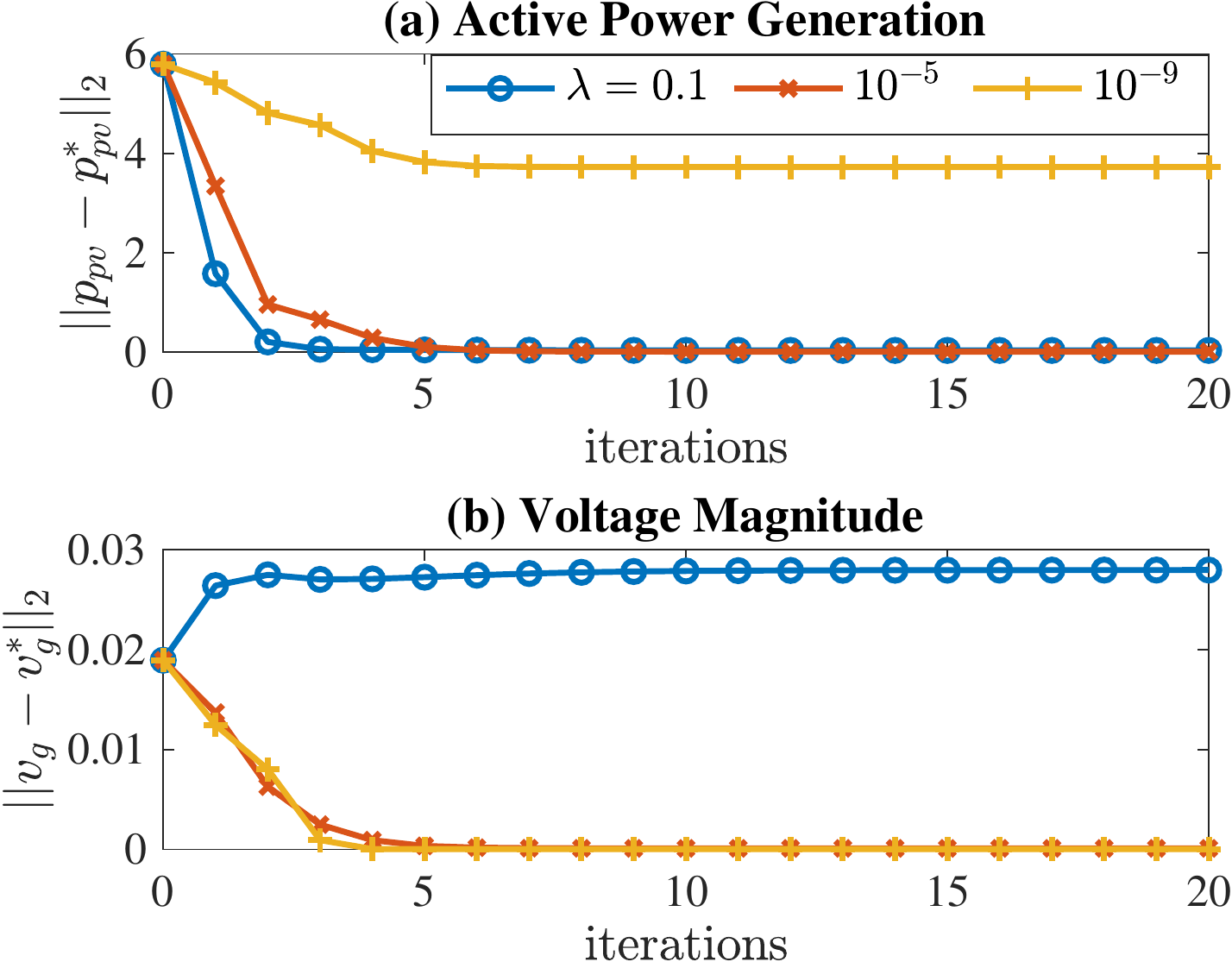}
	\caption{Convergence of the feasible path to the identified point for varying values of $\lambda$ for the 39-bus system.
}
	\label{fig:distmin_plot}
\end{figure}

\subsection{Convergence of Feasible Path Identification to a Specified Operating Point}
In this study, we investigate the convergence of the feasible path algorithm in an example based on the IEEE 39-bus system \cite{pai2012energy}.
The desired operating point was specified to be the globally optimal AC OPF solution. The initial operating point was determined by solving OPF with linear uniform generation cost (i.e., $c(p_g)=\sum_i p_{g,i}$) using M{\sc atpower}. The distance between the current and the desired operating point was minimized with different values of the parameter $\lambda$ in the objective in \eqref{eqn:FPI_obj}, 
which determines the trade-off between convergence for the active power and voltage magnitudes. 

Figure \ref{fig:distmin_plot} shows the convergence of the active power generation and voltage magnitudes to the desired operating point. For large enough values of $\lambda$, the active power outputs (at the non-slack generators) converge to those of the desired operating point. However, the voltage magnitudes may converge to a different, sub-optimal power flow solution.
Similarly, if $\lambda$ is set too low, the voltage magnitudes may converge to the desired operating point, while the active power set points do not. For intermediate values of $\lambda$, both active power and voltage magnitudes converge to the desired point.

The main takeaway from this result is that the convergence of the algorithm is path-dependent, i.e., there are cases where the sequential convex restriction gets trapped in a sub-optimal point. This issue could be mitigated by appropriate tuning of the objective function.

\subsection{Optimal Power Flow with Feasible Path Guarantees}
To show how the algorithm improves upon an initial, sub-optimal point (using the objective \eqref{eqn:pathguarantees} based on generation cost), we run our algorithm on different PGLib test cases, considering both the typical and congested operating conditions.
The initial operating point was obtained by solving the OPF problem with a linear uniform generation cost ($\sum_i p_{g,i}$) in M{\sc atpower}. These solutions are far from the optima of the OPF problems with their original objective functions. Thus, this experimental setup adequately exercises our sequential convex restriction algorithm.

\cref{tab:result} summarizes the numerical studies where the cost of generation was minimized at each iteration of the algorithm. 
The costs after the first and last iteration of the sequential convex restriction are shown in the fourth and fifth columns of \cref{tab:result}. The optimality gap at the $i$th iteration is defined as 
\begin{equation}
    \text{Optimality Gap}=\textstyle\frac{c(p_\textrm{g}^{(i)})-c(p_\textrm{g}^*)} {c(p_\textrm{g}^*)}.
\end{equation}
Here, $c(p_\textrm{g})$ is the objective function \eqref{eqn:cost}, $p_\textrm{g}^{(i)}$ is the power generation at the $i$th iteration and $p_\textrm{g}^*$ is the optimal set point from M{\sc atpower}. 
We observe that the OPF with convex restriction is able to significantly improve the operating point from the initial point, even if it does not reach the same solution as M{\sc atpower}. The algorithm converges within 5 iterations for all cases.
For all of the considered optimal power flow problems, the solutions obtained from the first step of the sequential convex restriction algorithm have optimality gaps less than 20\%, which suggests that the convex restriction covers a significant portion of the feasible space.
The average solver time per iteration is also shown in the last two columns. 

Note that the sequential convex restriction encountered numerical issues for three test cases (the 89- and 240-bus systems and 588-bus system with the congested operating conditions) which are omitted from \cref{tab:result}. 
These issues are due to the line flow constraints, which add quadratic limits on the existing quadratic envelopes. This introduces higher-order polynomial constraints (expressed in terms of two quadratic constraints), which can be numerically challenging. 

\begin{table*}[!htbp]
\centering
\caption{Optimality Gaps and Runtimes of OPF with Convex Restriction for Selected PGLib Test Cases}
\begin{tabular}{ |c||c|c|c|c|c|c|c|c|}
\hline
& \multicolumn{4}{c|}{Objective (\$/h)}  &  &  \multicolumn{2}{c|}{Optimality Gap (\%)} & \\
Test Case & Initial & M{\sc atpower} &  \shortstack{CVXRS \\ $1^\textrm{st}$ iter}  &\shortstack{CVXRS \\ last iter} & \shortstack{Number of \\ iterations} &  $1^\textrm{st}$ iter&  last iter & \shortstack{Solver Time \\ (seconds)} \\ \hline \hline
\multicolumn{9}{|c|}{Typical Operating Conditions (TYP)} \\ \hline
pglib\_opf\_case3\_lmbd & 6089.54 & 5812.64 & 5986.53 & 5813.54 & 5 & 2.99 & 0.02 & 0.01 \\ \hline
pglib\_opf\_case5\_pjm & 27356.2 & 17551.9 & 17839 & 17578.8 & 4 & 1.64 & 0.15 & 0.03 \\ \hline
pglib\_opf\_case14\_ieee & 7008.23 & 6291.28 & 6291.35 & 6291.29 & 2 & 0 & 0 & 0.08 \\ \hline
pglib\_opf\_case24\_ieee\_rts & 87065.8 & 63352.2 & 63393.8 & 63361.5 & 4 & 0.07 & 0.01 & 0.19 \\ \hline
pglib\_opf\_case30\_ieee & 12308.3 & 11974.5 & 11981.1 & 11976.8 & 2 & 0.06 & 0.02 & 0.23 \\ \hline
pglib\_opf\_case39\_epri & 152592 & 142980 & 144525 & 143010 & 4 & 1.08 & 0.02 & 0.41 \\ \hline
pglib\_opf\_case57\_ieee & 46216.5 & 39323.4 & 44000.3 & 42494 & 5 & 11.89 & 8.06 & 1.04 \\ \hline
pglib\_opf\_case73\_ieee\_rts & 262108 & 189764 & 189908 & 189789 & 5 & 0.08 & 0.01 & 1.73 \\ \hline
%pglib\_opf\_case89\_pegase & \multicolumn{8}{c|}{Numerical Issue $\dagger$}\\ \hline
pglib\_opf\_case118\_ieee & 145657 & 115804 & 117068 & 116071 & 5 & 1.09 & 0.23 & 3.55 \\ \hline
pglib\_opf\_case162\_ieee\_dtc & 129083 & 126154 & 127622 & 127612 & 3 & 1.16 & 1.16 & 19.91 \\ \hline
pglib\_opf\_case179\_goc & 905329 & 828404 & 893016 & 883301 & 5 & 7.8 & 6.63 & 11.47 \\ \hline
pglib\_opf\_case200\_tamu & 37398.7 & 34730.7 & 37138.3 & 35895.9 & 5 & 6.93 & 3.35 & 10.21 \\ \hline
pglib\_opf\_case300\_ieee & 850620 & 664220 & 734711 & 684909 & 5 & 10.61 & 3.11 & 55 \\ \hline
pglib\_opf\_case588\_sdet & 476950 & 381555 & 447566 & 428569 & 5 & 17.3 & 12.32 & 219.66 \\ \hline
\multicolumn{9}{|c|}{Congested Operating Conditions (API)} \\ \hline
pglib\_opf\_case3\_lmbd\_\_api & 11390.1 & 11242.1 & 11320.7 & 11242.4 & 4 & 0.7 & 0 & 0.01 \\ \hline
pglib\_opf\_case5\_pjm\_\_api & 83270.4 & 76377.4 & 76752 & 76433.2 & 4 & 0.49 & 0.07 & 0.06 \\ \hline
pglib\_opf\_case14\_ieee\_\_api & 13604.4 & 13310.7 & 13463.6 & 13424.1 & 3 & 1.15 & 0.85 & 0.1 \\ \hline
pglib\_opf\_case24\_ieee\_rts\_\_api & 282746 & 134948 & 241878 & 172528 & 5 & 79.24 & 27.85 & 0.31 \\ \hline
pglib\_opf\_case30\_ieee\_\_api & 24038.1 & 24032.1 & 24036.1 & 24036.1 & 1 & 0.02 & 0.02 & 0.34 \\ \hline
pglib\_opf\_case39\_epri\_\_api & 259792 & 257214 & 259405 & 258749 & 5 & 0.85 & 0.6 & 0.59 \\ \hline
pglib\_opf\_case57\_ieee\_\_api & 61522.6 & 59273.6 & 60600.3 & 60385.8 & 4 & 2.24 & 1.88 & 1.1 \\ \hline
pglib\_opf\_case118\_ieee\_\_api & 327478 & 316424 & 323357 & 318211 & 5 & 2.19 & 0.56 & 3.77 \\ \hline
pglib\_opf\_case162\_ieee\_dtc\_\_api & 144271 & 143514 & 144259 & 144259 & 1 & 0.52 & 0.52 & 11.79 \\ \hline
pglib\_opf\_case179\_goc\_\_api & 2456980 & 1971220 & 2381450 & 2330960 & 2 & 20.81 & 18.25 & 15.75 \\ \hline
pglib\_opf\_case200\_tamu\_\_api & 53307.9 & 44867.2 & 52408.5 & 51493.9 & 2 & 16.81 & 14.77 & 11.3 \\ \hline
pglib\_opf\_case300\_ieee\_\_api & 967348 & 775490 & 879185 & 841581 & 2 & 13.37 & 8.52 & 59.45 \\ \hline
\end{tabular}
\label{tab:result}
\vspace*{-1em}
\end{table*}

\section{Conclusion and Future Work}
\label{sec:conclusion}
This paper proposed a sequential convex restriction algorithm to obtain a \emph{feasible path} from an initial, feasible operating point to an improved operating point. The feasible path is a trajectory of dispatch points for which all points are AC power flow feasible and satisfy all operational constraints. The algorithm relies on solving a sequence of OPF problems that are formulated using convex restrictions, which are convex inner approximations of the OPF feasible space. 
The case studies demonstrate that the sequential convex restriction algorithm converges to a high-quality solution while generating feasible control actions, and is scalable to large systems.

One of the challenges we aim to address in future work is closing the optimality gaps when the sequential convex restriction algorithm is used to solve OPF problems.
Recent work in \cite{lee2019sequential} suggests that one of the necessary conditions for convergence of the sequential convex restriction algorithm is to have a matching Jacobian between the over- and under-estimator at the base point.
Constructing convex restrictions using alternative representations of the power flow equations may be also helpful for further closing the optimality gap.

There are other natural extensions to this work.
Our current approach requires a feasible initial operating point since a feasible path by its definition requires that all points on the path are feasible, including the initial point.
The method could be extended to allow for an infeasible initial point by relaxing the violated operational limits with slack variables.
The objective of the sequential convex restriction could then be designed to penalize violations of the operational limits such that the operating point finds a path back to a feasible dispatch point. This would be particularly useful to mitigate the impacts of contingencies that lead to an infeasible dispatch condition. 

Additionally, the current approach only considers steady-state security limits and does not consider system dynamics. We believe that it is possible to extend the method to consider system dynamics and plan to address this in the future.

We aim to further extend the approach by developing convex restrictions that also ensure N-1 security. Given multiple convex restrictions constructed for individual contingencies, a convex restriction that enforces the N-1 contingency criteria can be obtained by simply intersecting these restrictions due to the convexity of these sets.

We also plan to better understand cases where the method does not find a feasible path and consider additional control actions (e.g., changing the setpoints for FACTS devices and controllable transformers) that may help in these cases. 

Finally, our future work will analyze the number of control actions that are taken when following the feasible path, which is an important consideration for practical applications.

\appendix
Let $E_f\in\mathbb{R}^{n_b\times n_l}$ and $E_t\in\mathbb{R}^{n_b\times n_l}$ be the connection matrix for \textit{from} and \textit{to} buses. The $(k,l)^\mathrm{th}$ element of $E_f$ and the $(m,l)^\mathrm{th}$ element of $E_t$ are equal to 1 for each transmission line $l$, where the line $l$ connects the ``from" bus $k$ to ``to" bus $m$, and zero otherwise, and $E=E_f-E_t$. 
%From Kirchhoff's Current Law for the network shows
%\begin{equation*} S^\mathrm{inj}=\mathbf{diag}(V)\left(E_\textrm{f}I^\textrm{f}+E_\textrm{t}I^\textrm{t}+I_\textrm{sh}\right)^*. \end{equation*}
%where $I_f\in\mathbb{C}^{n_l}$ and $I_t\in\mathbb{C}^{n_l}$ denotes the current leaving from \textit{from} and \textit{to} bus, respectively. The current flow can be computed by
%\begin{equation*} \begin{aligned} I^\textrm{f}&=Y_\textrm{ff}V^\textrm{f}+Y_\textrm{ft}V^\textrm{t} \\ I^\textrm{t}&=Y_\textrm{tf}V^\textrm{f}+Y_\textrm{tt}V^\textrm{t} \end{aligned} \end{equation*}
%where $V^\textrm{f}\in\mathbb{C}^{n_l}$ and $V^\textrm{t}\in\mathbb{C}^{n_l}$ are complex voltage at \textit{from} and \textit{to} buses, respectively. 
The matrices $Y_\textrm{ff}$, $Y_\textrm{tf}$, $Y_\textrm{ft}$ and $Y_\textrm{tt}$ are diagonal with its elements,
\begin{displaymath} \begin{aligned}
Y_{\textrm{ff},\mathit{kk}}&=\left(y_k+j\frac{b^c_k}{2}\right)\left(\frac{1}{\tau_k}\right)^2, \hskip 1em
Y_{\textrm{tt},\mathit{kk}}=y_k+j\frac{b^c_k}{2}, \\
Y_{\textrm{ft},\mathit{kk}}&=-y_k\frac{1}{\tau_ke^{-j\theta^\textrm{shift}_k}}, \hskip 4em
Y_{\textrm{tf},\mathit{kk}}=-y_k\frac{1}{\tau_ke^{j\theta^\textrm{shift}_k}}.
\end{aligned}\end{displaymath}
where $Y_{\textrm{ff},\mathit{kk}}$ represents $k^\textrm{th}$ row and $k^\textrm{th}$ column of the diagonal matrix $Y_\textrm{ff}$. The transformer tap ratio, phase shift angle, and line charging susceptance are $\tau$, $\theta^\textrm{shift}$, $b^c$, respectively. These values are modeled as specified constants while solving the OPF problem.
%Note that the admittance matrix can be obtained by
%\begin{equation*} \begin{aligned} Y&=E_fY_\textrm{ff}E_f^T+E_tY_\textrm{tf}E_f^T \\ &\hskip 3em +E_fY_\textrm{ft}E_t^T+E_tY_\textrm{tt}E_t^T+Y_{sh}. \end{aligned} \end{equation*}
The phase adjusted admittance matrices $\widehat{Y}_\textrm{ft}\in\mathbb{C}^{n_l\times n_l}$ and $\widehat{Y}_\textrm{tf}\in\mathbb{C}^{n_l\times n_l}$ are diagonal matrices with
\begin{displaymath}
\widehat{Y}_\mathrm{ft}=Y_\mathrm{ft}\mathbf{diag}\left(e^{-j\varphi_0}\right), \hskip1em
\widehat{Y}_\mathrm{tf}=Y_\mathrm{tf}\mathbf{diag}\left(e^{j\varphi_0}\right).
\end{displaymath}
where each diagonal element of $\widehat{Y}_\textrm{ft}$ is an adjustment of $Y_\textrm{ft}$ by angle $\varphi_0$. The complex matrix $Y_{sh}$ is a diagonal matrix with the diagonal elements being the shunt admittances at the corresponding buses. Then,
\begin{displaymath} \begin{aligned}
\widehat{Y}^\textrm{c}&=E_\mathrm{f}\widehat{Y}_\mathrm{ft}+E_\mathrm{t}\widehat{Y}_\mathrm{tf},\ \ \widehat{Y}^\textrm{s}=E_\mathrm{f}\widehat{Y}_\mathrm{ft}-E_\mathrm{t}\widehat{Y}_\mathrm{tf}, \\
Y^\mathrm{d}&=E_\mathrm{f}Y_\mathrm{ff}E_\mathrm{f}^T+E_\mathrm{t}Y_\mathrm{tt}E_\mathrm{t}^T+Y_\mathrm{sh}.
\end{aligned} \end{displaymath}
Replacing $Y$ by $G$ or $B$ yields the real and imaginary parts of the $Y$ matrix, which are used in \eqref{eqn:pf_vec}, \eqref{eqn:lineflow_f}, and \eqref{eqn:lineflow_t}.

\ifCLASSOPTIONcaptionsoff
\newpage
\fi

\bibliographystyle{IEEEtran}
\bibliography{references}

\begin{thebibliography}{10}
\providecommand{\url}[1]{#1}
\csname url@samestyle\endcsname
\providecommand{\newblock}{\relax}
\providecommand{\bibinfo}[2]{#2}
\providecommand{\BIBentrySTDinterwordspacing}{\spaceskip=0pt\relax}
\providecommand{\BIBentryALTinterwordstretchfactor}{4}
\providecommand{\BIBentryALTinterwordspacing}{\spaceskip=\fontdimen2\font plus
\BIBentryALTinterwordstretchfactor\fontdimen3\font minus
  \fontdimen4\font\relax}
\providecommand{\BIBforeignlanguage}[2]{{%
\expandafter\ifx\csname l@#1\endcsname\relax
\typeout{** WARNING: IEEEtran.bst: No hyphenation pattern has been}%
\typeout{** loaded for the language `#1'. Using the pattern for}%
\typeout{** the default language instead.}%
\else
\language=\csname l@#1\endcsname
\fi
#2}}
\providecommand{\BIBdecl}{\relax}
\BIBdecl

\bibitem{carpentier1962}
J.~Carpentier, ``{Contribution a l'Etude du Dispatching Economique},''
  \emph{Bull. Soc. Franc. des Electriciens}, vol.~8, no.~3, pp. 431--447, 1962.

\bibitem{stott2012}
B.~Stott and O.~Alsa{\c{c}}, ``{Optimal Power Flow--Basic Requirements for
  Real-Life Problems and their Solutions},'' in \emph{12th Symp. Specialists in
  Electric Operational and Expansion Planning (SEPOPE)}, May 2012.

\bibitem{capitanescu2011stateoftheart}
F.~Capitanescu \emph{et~al.}, ``{State-of-the-Art, Challenges, and Future
  Trends in Security Constrained Optimal Power Flow},'' \emph{Electric Power
  Syst. Res.}, vol.~81, no.~8, pp. 1731 -- 1741, 2011.

\bibitem{opf_litreview1993IandII}
J.~A. Momoh, R.~Adapa, and M.~E. El-Hawary, ``{A Review of Selected Optimal
  Power Flow Literature to 1993. Parts I and II},'' \emph{IEEE Transactions on
  Power Systems}, vol.~14, no.~1, pp. 96--111, Feb. 1999.

\bibitem{ferc4}
A.~Castillo and R.~P. O'Neill, ``{Survey of Approaches to Solving the ACOPF
  (OPF Paper 4)},'' US Federal Energy Regulatory Commission, Tech. Rep., March
  2013.

\bibitem{Frank2012}
S.~Frank, I.~Steponavice, and S.~Rebennack, ``{Optimal Power Flow: A
  Bibliographic Survey I},'' \emph{Energy Syst.}, vol.~3, no.~3, pp. 221--258,
  Sept. 2012.

\bibitem{pscc2014survey}
P.~Panciatici \emph{et~al.}, ``{Advanced Optimization Methods for Power
  Systems},'' in \emph{18th Power Syst. Comput. Conf. (PSCC)}, Aug. 2014, pp.
  1--18.

\bibitem{abdi2017}
H.~Abdi, S.~D. Beigvand, and M.~L. Scala, ``{A Review of Optimal Power Flow
  Studies applied to Smart Grids and Microgrids},'' \emph{Renewable Sustainable
  Energy Rev.}, vol.~71, pp. 742--766, May 2017.

\bibitem{molzahn2018fnt}
D.~K. Molzahn and I.~A. Hiskens, ``{A Survey of Relaxations and Approximations
  of the Power Flow Equations},'' \emph{Found. Trends Electric Energy Syst.},
  vol.~4, no. 1-2, pp. 1--221, Feb. 2019.

\bibitem{hong1993}
Y.-Y. Hong, C.-M. Liao, and T.-G. Lu, ``{Application of Newton Optimal Power
  Flow to Assessment of VAR Control Sequences on Voltage Security: Case Studies
  for a Practical Power System},'' \emph{IEE Proc. C (Generation, Transmission
  and Distribution)}, vol. 140, pp. 539--544, Nov. 1993.

\bibitem{Capitanescu2011}
F.~Capitanescu and L.~Wehenkel, ``{Redispatching Active and Reactive Powers
  using a Limited Number of Control Actions},'' \emph{IEEE Trans. Power Syst.},
  vol.~26, no.~3, pp. 1221--1230, Aug. 2011.

\bibitem{phan2014minimal}
D.~T. Phan and X.~A. Sun, ``{Minimal Impact Corrective Actions in
  Security-Constrained Optimal Power Flow via Sparsity Regularization},''
  \emph{IEEE Trans. Power Syst.}, vol.~30, no.~4, pp. 1947--1956, 2014.

\bibitem{bukhsh_tps}
W.~A. Bukhsh, A.~Grothey, K.~I.~M. McKinnon, and P.~A. Trodden, ``{Local
  Solutions of the Optimal Power Flow Problem},'' \emph{IEEE Trans. Power
  Syst.}, vol.~28, no.~4, pp. 4780--4788, 2013.

\bibitem{Molzahn2017}
D.~K. Molzahn, ``{Computing the Feasible Spaces of Optimal Power Flow
  Problems},'' \emph{IEEE Trans. Power Syst.}, vol.~32, no.~6, pp. 4752--4763,
  Nov. 2017.

\bibitem{Lee2018}
D.~{Lee}, H.~D. {Nguyen}, K.~{Dvijotham}, and K.~{Turitsyn}, ``{Convex
  Restriction of Power Flow Feasibility Sets},'' \emph{IEEE Trans. Control
  Netw. Syst.}, vol.~6, no.~3, pp. 1235--1245, Sept. 2019.

\bibitem{garcia81a}
C.~B. Garcia and W.~I. Zangwill, \emph{{Pathways to Solutions, Fixed Points and
  Equilibria}}.\hskip 1em plus 0.5em minus 0.4em\relax Englewood Cliffs, NJ:
  Prentice Hall, 1981.

\bibitem{continuation}
V.~{Ajjarapu} and C.~{Christy}, ``{The Continuation Power Flow: A Tool for
  Steady State Voltage Stability Analysis},'' \emph{IEEE Trans. Power Syst.},
  vol.~7, no.~1, pp. 416--423, Feb. 1992.

\bibitem{ieee02a}
I.~Dobson \emph{et~al.}, ``{Voltage Stability Assessment: Concepts, Practices
  and Tools},'' IEEE Power Engineering Society, Power System Stability
  Subcommittee Special Publication SP101PSS, Aug. 2002.

\bibitem{milano2010}
F.~Milano, \emph{{Continuation Power Flow Analysis}}.\hskip 1em plus 0.5em
  minus 0.4em\relax Berlin, Heidelberg: Springer Berlin Heidelberg, 2010, pp.
  103--130.

\bibitem{hiskens_boundary}
I.~A. {Hiskens} and R.~J. {Davy}, ``{Exploring the Power Flow Solution Space
  Boundary},'' \emph{IEEE Trans. Power Syst.}, vol.~16, no.~3, pp. 389--395,
  Aug. 2001.

\bibitem{thorp1993}
W.~{Ma} and J.~S. {Thorp}, ``{An Efficient Algorithm to Locate All the Load
  Flow Solutions},'' \emph{IEEE Trans. Power Syst.}, vol.~8, no.~3, pp.
  1077--1083, Aug. 1993.

\bibitem{molzahn_lesieutre_chen-counterexample}
D.~K. Molzahn, B.~C. Lesieutre, and H.~Chen, ``{Counterexample to a
  Continuation-Based Algorithm for Finding All Power Flow Solutions},''
  \emph{IEEE Trans. Power Syst.}, vol.~28, no.~1, pp. 564--565, Feb. 2013.

\bibitem{lesieutre2015}
B.~C. {Lesieutre} and D.~{Wu}, ``{An Efficient Method to Locate All the Load
  Flow Solutions--Revisited},'' in \emph{53rd Annu. Allerton Conf. Commun.,
  Control, Comput. (Allerton)}, Sept. 2015, pp. 381--388.

\bibitem{ponrajah1989}
R.~A. {Ponrajah} and F.~D. {Galiana}, ``{The Minimum Cost Optimal Power Flow
  Problem Solved via the Restart Homotopy Continuation Method},'' \emph{IEEE
  Trans. Power Syst.}, vol.~4, no.~1, pp. 139--148, Feb. 1989.

\bibitem{huneault1990investigation}
M.~Huneault and F.~Galiana, ``{An Investigation of the Solution to the Optimal
  Power Flow Problem incorporating Continuation Methods},'' \emph{IEEE Trans.
  Power Syst.}, vol.~5, no.~1, pp. 103--110, 1990.

\bibitem{cvijic2012}
S.~{Cvijic}, P.~{Feldmann}, and M.~{Hie}, ``{Applications of Homotopy for
  solving AC Power Flow and AC Optimal Power Flow},'' in \emph{IEEE Power and
  Energy Society General Meeting (PESGM)}, July 2012, pp. 1--8.

\bibitem{Bie18Toward}
P.~{Bie}, H.~{Chiang}, B.~{Zhang}, and N.~{Zhou}, ``{Toward Online Multi-Period
  Power Dispatch with AC Constraints and Renewable Energy},'' \emph{IET
  Generation, Transmission \& Distribution}, vol.~12, no.~14, pp. 3502--3509,
  2018.

\bibitem{wu_molzahn_lesieutre_dvijotham-acopf_tracing}
D.~Wu, D.~K. Molzahn, B.~C. Lesieutre, and K.~Dvijotham, ``{A Deterministic
  Method to Identify Multiple Local Extrema for the AC Optimal Power Flow
  Problem},'' \emph{IEEE Trans. Power Syst.}, vol.~33, no.~1, pp. 654--668,
  Jan. 2018.

\bibitem{huneault1985continuation}
M.~{Huneault}, A.~{Fahmideh-Vojdani}, M.~{Juman}, R.~{Calderon}, and F.~D.
  {Galiana}, ``{The Continuation Method in Power System Optimization:
  Applications to Economy-Security Functions},'' \emph{IEEE Trans. Power Appar.
  Syst.}, vol. PAS-104, no.~1, pp. 114--124, Jan. 1985.

\bibitem{almeida1994}
K.~C. {Almeida}, F.~D. {Galiana}, and S.~{Soares}, ``{A General Parametric
  Optimal Power Flow},'' \emph{IEEE Transa. Power Syst.}, vol.~9, no.~1, pp.
  540--547, Feb. 1994.

\bibitem{ajjarapu1995optimal}
V.~Ajjarapu and N.~Jain, ``{Optimal Continuation Power Flow},'' \emph{Electric
  Power Syst. Res.}, vol.~35, no.~1, pp. 17--24, 1995.

\bibitem{almeida2000}
K.~C. {Almeida} and R.~{Salgado}, ``{Optimal Power Flow Solutions under
  Variable Load Conditions},'' \emph{IEEE Trans. Power Syst.}, vol.~15, no.~4,
  pp. 1204--1211, Nov. 2000.

\bibitem{mpc_nerc}
N.~{Atic}, D.~{Rerkpreedapong}, A.~{Hasanovic}, and A.~{Feliachi}, ``{NERC
  Compliant Decentralized Load Frequency Control Design using Model Predictive
  Control},'' in \emph{IEEE Power Engineering Society General Meeting (PESGM)},
  vol.~2, July 2003, pp. 554--559.

\bibitem{mpc_agc}
A.~N. {Venkat}, I.~A. {Hiskens}, J.~B. {Rawlings}, and S.~J. {Wright},
  ``{Distributed MPC Strategies With Application to Power System Automatic
  Generation Control},'' \emph{IEEE Trans. Control Syst. Tech.}, vol.~16,
  no.~6, pp. 1192--1206, Nov. 2008.

\bibitem{moradzadeh2012}
M.~Moradzadeh, R.~Boel, and L.~Vandevelde, ``{Voltage Coordination in
  Multi-Area Power Systems via Distributed Model Predictive Control},''
  \emph{IEEE Trans. Power Syst.}, vol.~28, no.~1, pp. 513--521, 2012.

\bibitem{fuchs2013}
A.~Fuchs, M.~Imhof, T.~Demiray, and M.~Morari, ``{Stabilization of Large Power
  Systems using VSC--HVDC and Model Predictive Control},'' \emph{IEEE Trans.
  Power Del.}, vol.~29, no.~1, pp. 480--488, 2013.

\bibitem{martin2016}
J.~A. Martin and I.~A. Hiskens, ``{Corrective Model-Predictive Control in Large
  Electric Power Systems},'' \emph{IEEE Trans. Power Syst.}, vol.~32, no.~2,
  pp. 1651--1662, 2016.

\bibitem{Bernstein2015}
A.~Bernstein, L.~Reyes-Chamorro, J.-Y. {Le Boudec}, and M.~Paolone, ``{A
  Composable Method for Real-Time Control of Active Distribution Networks with
  Explicit Power Setpoints. Part I: Framework},'' \emph{Electric Power Syst.
  Res.}, vol. 125, pp. 254--264, Aug. 2015.

\bibitem{Wang2017}
C.~Wang, J.-Y. {Le Boudec}, and M.~Paolone, ``{Controlling the Electrical State
  via Uncertain Power Injections in Three-Phase Distribution Networks},''
  \emph{IEEE Trans. Smart Grid}, vol.~10, no.~2, pp. 1349--1362, March 2019.

\bibitem{conn2000trust}
A.~R. Conn, N.~I.~M. Gould, and P.~L. Toint, \emph{{Trust Region
  Methods}}.\hskip 1em plus 0.5em minus 0.4em\relax SIAM, 2000, vol.~1.

\bibitem{nocedal2006numerical}
J.~Nocedal and S.~Wright, \emph{{Numerical Optimization}}.\hskip 1em plus 0.5em
  minus 0.4em\relax Springer Science {\&} Business Media, 2006.

\bibitem{zamzam2017}
A.~S. Zamzam, C.~Zhao, E.~Dall'Anese, and N.~D. Sidiropoulos, ``{A QCQP
  Approach for OPF In Multiphase Radial Networks with Wye and Delta
  Connections},'' in \emph{10th IREP Symp. Bulk Power Syst. Dynamics Control},
  August 2017.

\bibitem{wei2017}
W.~Wei, J.~Wang, N.~Li, and S.~Mei, ``{Optimal Power Flow of Radial Networks
  and Its Variations: A Sequential Convex Optimization Approach},'' \emph{IEEE
  Trans. Smart Grid}, vol.~8, no.~6, pp. 2974--2987, Nov. 2017.

\bibitem{brouwer1911abbildung}
L.~E.~J. Brouwer, ``{{\"U}ber Abbildung von Mannigfaltigkeiten},''
  \emph{Mathematische Annalen}, vol.~71, no.~1, pp. 97--115, 1911.

\bibitem{pglib}
\BIBentryALTinterwordspacing
{IEEE PES Task Force on Benchmarks for Validation of Emerging Power System
  Algorithms}, ``{The Power Grid Library for Benchmarking {AC} Optimal Power
  Flow Algorithms},'' \emph{arXiv:1908.02788}, Aug. 2019. [Online]. Available:
  \url{https://github.com/power-grid-lib/pglib-opf}
\BIBentrySTDinterwordspacing

\bibitem{jump}
I.~Dunning, J.~Huchette, and M.~Lubin, ``{JuMP: A Modeling Language for
  Mathematical Optimization},'' \emph{SIAM Rev.}, vol.~59, no.~2, 2017.

\bibitem{zimmerman11}
R.~D. Zimmerman, C.~E. Murillo-S{\'{a}}nchez, and R.~J. Thomas,
  ``{\mbox{MATPOWER:} Steady-State Operations, Planning, and Analysis Tools for
  Power Systems Research and Education},'' \emph{IEEE Trans. Power Syst.},
  vol.~26, no.~1, pp. 12--19, 2011.

\bibitem{chow1982time}
J.~H. Chow, G.~Peponides, P.~V. Kokotovic, B.~Avramovic, and J.~R. Winkelman,
  \emph{{Time-Scale Modeling of Dynamic Networks with Applications to Power
  Systems}}.\hskip 1em plus 0.5em minus 0.4em\relax Springer, 1982, vol.~46.

\bibitem{pai2012energy}
M.~A. Pai, \emph{{Energy Function Analysis for Power System Stability}}.\hskip
  1em plus 0.5em minus 0.4em\relax Springer Science \& Business Media, 2012.

\bibitem{lee2019sequential}
D.~Lee, K.~Turitsyn, and J.-J. Slotine, ``{Sequential Convex Restriction and
  its Applications in Robust Optimization},'' \emph{arXiv:1909.01778}, 2019.

\end{thebibliography}
%\nocite{*}

\end{document}